\documentclass[12pt]{amsart}
\usepackage{amsmath,amsthm,amsfonts,amssymb,amscd,mathrsfs}
\usepackage[margin=3cm]{geometry}
\usepackage{color}

\usepackage{xcolor,cancel}

\UseRawInputEncoding
\usepackage{ulem}
\usepackage{bbm}

\usepackage{comment}
\usepackage{graphicx, graphics}
\usepackage[pdftex]{pict2e}
\usepackage{cases}
\usepackage{todonotes}
\usepackage{psfrag,pst-node,subfigure,rotating,
amsmath, bbm, amsthm, amssymb, amsthm, setspace, picture, epsfig,
amsfonts}
\usepackage{multirow}
\usepackage{indentfirst}
\usepackage{arydshln}
\usepackage[colorlinks,linkcolor=blue,anchorc
olor=green,citecolor=blue]{hyperref}
\usepackage{hyperref}
\numberwithin{equation}{section}
\newcommand{\bea}{\begin{eqnarray}}
\newcommand{\eea}{\end{eqnarray}}
\newcommand{\Bea}{\begin{eqnarray*}}
\newcommand{\Eea}{\end{eqnarray*}}

\theoremstyle{plain}

\newtheorem{Thm}{Theorem}[section]
\newtheorem{Lem}[Thm]{Lemma}

\newtheorem{Prop}[Thm]{Proposition}
\newtheorem{Cor}[Thm]{Corollary}
\newtheorem{Def}[Thm]{Definition}

\newtheorem{Rem}[Thm] {Remark}
\theoremstyle{remark}

\long\def\begcom#1\endcom{}

\newcommand{\Sep}{\operatorname{Sep}}
\newcommand{\wSep}{\widehat{{\rm Sep}}}
\newcommand{\rd}{{\rm d}}

\newcounter{RomanNumber}

\def \N{{\mathbb N}}

\def\R{          \mathbb R}

\def\DS{\displaystyle}

\def\ln{\operatorname{ln}}

\def\brA{{\bar A}}
\def\brC{{\bar C}}

\def\bA{\bar A}

\def\brtheta{{\bar\theta}}

\def\bS{{\mathbf{S}}}
\def\bv{{\mathbf{v}}}

\def\cA{{\mathcal A}}
\def\cD{{\mathcal D}}

\def\cU{{\mathcal U}}

\def\diam{{\rm diam}}
\def\tB{{\tilde B}}
\def\bsigma{\bar{\sigma}}
\def\bfE{\bar{E}}

\def\fv{\sigma}

\def\hB{{\hat{B}}}
\definecolor{bluegray}{rgb}{0.1, 0.1, 0.6}
\definecolor{dgreen}{rgb}{0.1,0.6,0.1}

\definecolor{bluegreen}{rgb}{0.1,0.5,0.2}%{0.6,0.3,0.4}
\definecolor{bpurple}{rgb}{0.74,0.2,0.64}%{0.6,0.3,0.4}

\definecolor{orange}{rgb}{0.8, 0.33, 0.0}

\usepackage{mathabx}

\def\EXP{{\mathbb{E}}}
\def\Prob{{\mathbb{P}}}

\def\eps{{\varepsilon}}
\definecolor{deepcarrotorange}{rgb}{1, 0.31, 0.0}%{0.91, 0.41, 0.17}

\title[Generalized multiple BC Lemma in dynamics and its applications.]{Generalized multiple Borel-Cantelli Lemma in dynamics and its applications}

\author{Sixu Liu}
\address[S.~Liu]{Beijing Institute of Mathematical Sciences and Applications\\
 Beijing 101408, China}
\email{liusixu@bimsa.cn}
\begin{document}
\maketitle
\begin{abstract}
Multiple Borel-Cantelli Lemma is a criterion that characterizes the occurrence of multiple rare events on the same time scale. We generalize the multiple Borel-Cantelli Lemma in dynamics established  by Dolgopyat, Fayad and Liu [J. Mod. Dyn. {\bf 18} (2022) 209--289], broadening its applications to encompass several non-smooth systems  with absolute  continuous measures.  Utilizing this generalization, we  derive multiple Logarithm Law for  hitting time and recurrence of  dispersing billiard maps and piecewise expanding maps under some regular conditions, including tent map, Lorentz-like map and Gauss map.
 \end{abstract}

\maketitle
%\tableofcontents
\section{Introduction}
Borel-Cantelli Lemma is primarily used to study the occurrence of events as a fundamental tool in probability theory. Let $(\Omega, \mathcal{F}, \mathbb{P})$ be a probability space and $E_n\subset\Omega$ be a sequence of measurable sets. A classical Borel-Cantelli Lemma states:
\begin{itemize}
 \item[(i)] If $\sum_{n=1}^{\infty} \Prob\left(E_n\right)<\infty,$ then almost every point belongs to finitely many $E_n.$
   \item[(ii)] If $\sum_{n=1}^{\infty} \Prob\left(E_n\right)=\infty$ and the events $\{E_n\}_{n=1}^\infty$ are independent, then almost every point belongs to infinitely many $E_n.$
\end{itemize}
Since the events in dynamical systems are obtained by iterations and often lack independence, for Borel-Cantelli Lemmas in dynamical systems, we replace the independence condition with weakly dependent conditions.

Research on dynamical Borel-Cantelli Lemmas and their applications for recurrence has yielded rich results.
One of the first Borel-Cantelli lemma for dynamical systems was established by Philipp  for certain expanding maps on the unit interval~\cite{Philipp67}. Schmidt gave conditions for a sequence of non-independent events to satisfy the Borel-Cantelli lemma, with a context related to Diophantine approximations~\cite{Spr79}. For Anosov diffeomorphisms, Chernov and Kleinbock~\cite{CK06} obtained the dynamical Borel-Cantelli lemma, while Barreira and Saussol~\cite{BS01} established Logarithm Law for recurrence. The corresponding results for partially hyperbolic systems were extended by Dolgopyat~\cite{Dolgopyat04}. Dynamical Borel-Cantelli Lemmas  have also been proven for some non-smooth systems, including uniformly expanding maps and nonuniformly expanding maps~\cite{FMP12,GNO10,KZ23}.

Dynamical Borel-Cantelli Lemmas  serve as  fundamental tools for characterizing  first-time recurrence. To provide a quantitative description of multiple recurrence, we require  multiple Borel-Cantelli Lemma. This criterion determines the occurrence of multiple rare events on the same time scale.  Considering a family of events $\{E_{\rho_n}^k\}_{1\leq k\leq n}$ and $r\geq1,$ multiple Borel-Cantelli Lemma provides conditions under which the events $E_{\rho_n}^k,\,1\leq k\leq n$ occur at least $r$ times for infinitely many $n.$ This  facilitates diverse applications in addressing recurrence problems of dynamical systems and number theory.

In the independent case, multiple Borel-Cantelli Lemma was originally obtained by Mori~\cite{Mo76}, and subsequent work by Aaronson and  Nakada~\cite{AN} explored its applicability to systems with good symbolic dynamics. For  broader dynamical system applications, Dolgopyat, Fayad and Liu~\cite{DFL22} constructed a comprehensive framework for establishing  multiple Borel-Cantelli Lemma. The results imply multiple Logarithm Laws  for recurrence and hitting times, as well as Poisson Limit Laws for systems that are exponentially  mixing of all orders.  Furthermore, Dolgopyat and Liu~\cite{DL23} explored the application of multiple Borel-Cantelli Lemma to the study of  heavy-tailed random variables and established  an analogue of the Law of Iterated Logarithm.

The multiple Borel-Cantelli Lemma established in~\cite{DFL22,DL23}  has found applications in the study of smooth systems with smooth measures. However, when dealing with systems that exhibit singularities, the dynamics become more intricate, presenting additional challenges for researchers. This paper aims to  extend  multiple Borel-Cantelli Lemma for broader applicability across various systems. Furthermore, we apply the generalized multiple Borel-Cantelli Lemma to characterize multiple hitting times and recurrence  for dispersing billiard maps and piecewise expanding maps under some regularity conditions, including tent map, Lorentz-like map and Gauss map.

The generalized multiple Borel-Cantelli Lemma presents a relaxation of the mixing conditions that were previously established in~\cite{DFL22}. Specifically, it requires that the probabilities of the occurrence of separated events are bounded by a multiple of the product of their individual probabilities, rather than being approximated directly by the product. Such a refinement broadens the scope of the multiple Borel-Cantelli Lemma, making it applicable to  non-smooth systems with absolutely continuous invariant measures, as opposed to the restrictive context of smooth maps with smooth measures.

In addition, \cite{DFL22} established a criterion used to analyse rare events characterized by visits to a sublevel set of a Lipschitz function by trajectories of a smooth system  that is multiple exponential mixing for Lipschitz functions.  However, non-smooth systems, such as dispersing billiard maps and piecewise expanding maps, do not exhibit exponential mixing properties for Lipschitz functions equipped with the Lipschitz norm. To address this, we introduce the spaces of dynamically H\"older continuous functions and functions with bounded variation, each equipped with corresponding norms. Multiple exponential mixing properties of these spaces enable us to establish multiple Logarithm Law for hitting times and recurrence.

The remainder of this paper is organized as follows. In Section \ref{ScGMBC}, we describe  generalized multiple Borel-Cantelli Lemma and explore its applications to multiple exponentially mixing systems. In Section \ref{SSMultExpMixBil}, we establish   multiple Logarithm Law for hitting time and recurrence of dispersing billiard maps. The corresponding results for  piecewise expanding maps are presented in Section \ref{SSMultExpMixPem}. Finally, for the sake of completeness,  we include some standard proofs as appendices.

\section{Generalized multiple Borel-Cantelli Lemma and applications.}\label{ScGMBC}
\subsection{Generalized multiple Borel-Cantelli Lemma.}
Let $(\Omega, \mathcal{F}, \mathbb P)$ be a probability space, $\{\rho_n\}_{n=1}^\infty$ be a sequence of positive numbers, and $\{ E_{\rho_{n}}^k \}_{(n,k) \in \N^2; 1\leq k \leq 2n}$ be a family of events.
Given $r\in\mathbb{Z}_+,$ we denote
$$N^n_{\rho_n}(\omega)=\sharp\left\{1\leq k\leq n:\,\,\omega\in E_{\rho_{n}}^k \right\}$$
and $H_r$ the set of points such that $E^k_{\rho_n}$ occur at least $r$ times for infinitely many $n,$ i.e. $$H_r=\bigcap_{m=1}^\infty\bigcup_{n=m}^\infty\left\{\omega\in\Omega:N_{\rho_n}^n(w)\geq r\right\}.$$
Our aim is to establish a criterion for which the set $H_r$ has either zero or positive measure. To achieve this, we introduce several conditions that quantify the asymptotic independence between the events $\{ E_{\rho_{n}}^k \}_{(n,k) \in \N^2; 1\leq k \leq 2n}.$

Let $s$ and $\hat{s}$ be functions $\mathbb{Z}_+\rightarrow\mathbb{R}_+$ such that
$$s(n) \leq C(\ln n)^2 \quad \text{and} \quad \varepsilon n\leq \hat{s}(n) <\frac{n(1-q)}{2r}$$ for some constants $C>0,$ $0<q<1$ and $0<\varepsilon<(1-q)/{(2r)}.$ We employ $\Sep$ and  $ \wSep$ as the {\it separation indices} to represent the separated time of the occurrence of the events $\{ E_{\rho_{n}}^k \}_{(n,k) \in \N^2; 1\leq k \leq 2n}$:
\begin{align*}\Sep_n(k_1,\dots, k_r)=\sharp\left\{0\leq j\leq r-1: k_{j+1}-k_j\geq s(n)\right\},\\
 \wSep_n(k_1,\dots, k_r)=\sharp\left\{0\leq j\leq r-1: k_{j+1}-k_j\geq \hat{s}(n)\right\}\end{align*}
for any $r$-tuple $0<k_1<k_2\dots<k_r \leq n,$ where $k_0=0.$

Suppose $\sigma:\R_+ \to \R_+$ represents a monotonically  function, and $C$ is a positive constant. We introduce the following conditions for the family of events $\{ E_{\rho_{n}}^k \}_{(n,k) \in \N^2; 1\leq k \leq 2n}.$
Conditions $(GM1)_{r}$ and $(GM3)_{r}$ are imposed as mixing conditions, while condition $(GM2)_{r}$​ serves as a non-clustering condition.

\begin{itemize}

\item[$(GM1)_r$]
If $0< k_1<k_2<\dots k_r\leq n$ are such that $\Sep_n(k_1, \dots, k_r)=r$ then
$$
C^{-1}\fv(\rho_n)^r \leq
  \Prob\left(\bigcap_{j=1}^r E^{k_j}_{\rho_n} \right)\leq C\fv(\rho_n)^r.$$

\item[$(GM2)_r$] If $0< k_1<k_2<\dots k_r\leq n$ are such that $\Sep_n(k_1, \dots, k_r)=m<r$, then
$$  \Prob\left(\bigcap_{j=1}^r E^{k_j}_{\rho_n} \right)
\leq {\frac{C \fv(\rho_n)^m}{(\ln n)^{100r}}}.$$
%\frac{K}{n^{m+\eta}}.
\item[$(GM3)_r$]  If $0<k_1<k_2<\dots<k_r<l_1<l_2<\dots<l_r$ are such that  $2^i<k_\alpha\leq 2^{i+1}, 2^j<l_\beta\leq 2^{j+1}$ for $1\leq\alpha,\beta\leq r$, $j-i\geq b$ for some constant $b\geq 1,$ and such that
$$\wSep_{2^{i+1}}(k_1,\ldots,k_r)=r, \quad \wSep_{2^{j+1}}(l_1,\ldots,l_r)=r,
 \quad l_1-k_r\geq \hat{s}(2^{j+1}),$$
then
$$
\Prob\left(\left(\bigcap _{\alpha=1}^r E^{k_\alpha}_{\rho_{2^i}}\right)\bigcap
\left( \bigcap_{\beta=1}^r E^{l_\beta}_{\rho_{2^j}} \right)\right)\leq
C\fv(\rho_{2^i})^r\fv(\rho_{2^j})^r.$$
\end{itemize}

\begin{Thm}\label{GMultiBCN}
  Given a family of events $\{ E_{\rho_{n}}^k \}_{(n,k) \in \N^2; 1\leq k \leq 2n}$ and define
$$ \bS_r=\sum_{j=1}^\infty 2^{rj} {\fv(\rho_{2^j})^r}. $$
\begin{itemize}
\item[(a)] Suppose $E_{\rho_{n_1}}^k\subset E_{\rho_{n_2}}^k$ and $\fv(\rho_{n_1})\leq\fv(\rho_{n_2})$ for $n_1\geq n_2.$ If $\bS_r<\infty,$ $(GM1)_r$ and
$(GM2)_r$ are satisfied, then $\Prob(H_r)=0.$
\item[(b)] If $\bS_r=\infty,$  $(GM1)_k$ and $(GM2)_k$ are satisfied for $k=r,r+1,\dots, 2r$ and $(GM3)_r$ is satisfied,
then $\Prob(H_r)>0.$
\end{itemize}
\end{Thm}

The proof  of Theorem~\ref{GMultiBCN} proceeds in a similar fashion to the proof of Theorem~2.4 in \cite{DFL22} with slight modifications. For the reader's convenience we include the argument in Appendix~\ref{appmbc}.

\begin{Rem}
The mixing conditions $(GM1)_{r}$ and $(GM3)_{r}$ in our context are significantly weaker than $(M1)_{r}$ and $(M3)_{r}$ presented in \cite{DFL22}. Condition $(GM1)_r$ asserts that the probability $\Prob\left(\bigcap_{j=1}^r E^{k_j}_{\rho_n}\right)$ can be bounded by a multiple of $\sigma(\rho_n)^r,$ whereas the condition $(M1)_r$ in \cite{DFL22} requires that $\Prob\left(\bigcap_{j=1}^r E^{k_j}_{\rho_n}\right)$ should approximate $\sigma(\rho_n)^r$ for sufficiently large $n.$ We extend a similar generalization to condition $(GM3)_r$ compared to condition $(M3)_r.$  For any given system $(X,\mathcal{B},\mu,f),$ $(GM1)_{r}$ and $(GM3)_{r}$ are applicable not only to smooth map $f$ with a smooth measure $\mu,$ but also to some non-smooth systems with absolutely continuous invariant measures.
\end{Rem}

\begin{Rem}[About notations]
\begin{itemize}
\item[(1)] Throughout this paper, we use $C$  to denote a constant depending on the fixed number $r$ of the occurring events,
which may change from line to line but remains independent of the  time scale $n,$ the events $E_{\rho}$ and $\bfE_{\rho},$ the order of iteration of $f,$ etc.

\item[(2)] For functions $\alpha$ and $\beta$ defined on $\mathbb{N}$ or $\mathbb{R},$ and  taking values in a normed space, we write $\beta=O(\alpha)$ if there exists some constant $C>0$ such that $\|\beta(x)\|\leq C\|\alpha(x)\|$ for any $x.$

\item[(3)] Let $(X,\mathcal{B},\mu)$ be a measure space. For a measurable function $\phi:X^k\rightarrow\mathbb{R},$ $k\in\mathbb{Z}_+,$ we denote the integral of $\phi$ by
$$\mathbb{E}\phi=\int_{X^k} \phi(x_1,\cdots,x_k)\rd\mu(x_1)\cdots \rd\mu(x_k).$$
\end{itemize}
\end{Rem}

\subsection{Multiple exponentially mixing systems.}\label{MEMS}
We present how to obtain conditions $(GM1)_r$ $(GM2)_r$ and $(GM3)_r$ by imposing specific conditions on systems and targets. Let $(X,\mathcal{B},\mu,f)$ be a measure preserving system. To deal with multiple hitting and recurrence, we consider two families of sets:
$$E_\rho \subseteq X,\,\,\,{\bfE}_{\rho}\subseteq X\times X,\,\,\,\rho \in \R_+.$$
We take $$E_\rho^k=f^{-k}E_\rho,\quad{\bfE}^k_{\rho}=\{x: (x, f^kx)\in{\bfE}_{\rho}\}$$ and
$$\fv(\rho)=\mu(E_\rho),\quad\bar{\fv}(\rho)
=(\mu\times\mu)({\bfE}_{\rho}).$$

From now on we will always assume that if $\rho'\leq \rho$, then
$$E_{\rho'}\subset E_\rho, \quad \bfE_{\rho'}\subset \bfE_\rho.$$

The multiple exponential mixing properties of the dynamical system $(X,\mathcal{B},\mu,f)$ and the regularity and shrinking  conditions outlined below yield $(GM1)_r$, $(GM2)_r$ and $(GM3)_r.$

\begin{Def}[Generalized $(r+1)$-fold exponentially mixing systems]\label{def.mixing}
Let $\mathbb{B}$ be  a space of real
valued functions  defined over $X^{r+1}$, with a norm $\|\cdot\|_\mathbb{B}$. For $r \geq 1$, we say that $(X,\mathcal{B},\mu,f,\mathbb{B})$ is a generalized $(r+1)$-fold exponentially mixing system,  if there exist constants $C>0, L>0$ and $\theta<1$ such that
\begin{itemize}
\item[(Prod)]$\displaystyle{ \|A_1 A_2\|_\mathbb{B} \leq C \|A_1\|_\mathbb{B} \|A_2\|_\mathbb{B},}$
\item[(Gr)]$\displaystyle{  \|A\circ (f^{k_0},\ldots,f^{k_{r}}) \|_\mathbb{B} \leq C L^{\sum_{i=0}^{r} k_i} \|A\|_\mathbb{B},}$
\item[$({\rm GEM})_r$] If $0=k_0\leq k_1\leq \ldots \leq k_r$ are such that  $\forall i\in [0,r-1], k_{i+1}-k_i \geq m$, then
    $$C^{-1}\left(\int_{X^{r+1}}  A(x_0,\cdots,x_r)\rd\mu(x_0)\cdots \rd\mu(x_r)-\theta^m \left\|A\right\|_{\mathbb{B}}\right)\leq$$
 $$\int_{X} A(x,f^{k_1}x,\cdots,f^{k_r}x) \rd\mu(x)\leq C\left(\int_{X^{r+1}}  A(x_0,\cdots,x_r)\rd\mu(x_0)\cdots \rd\mu(x_r)+ \theta^m \left\|A\right\|_{\mathbb{B}}\right).
$$

\end{itemize}
\end{Def}

Given a system $(X,\mathcal{B},\mu,f,\mathbb{B}),$ we introduce the concept of simple and composite admissible targets.

\begin{Def}[Simple admissible targets]\label{def.targets} Let $\{E_\rho\}_{\rho\in \mathbb{R}_+}$ be a collection of sets in $X$
 for which there are positive constants $\tau,\,C$ and $\eta$ such that for all
  sufficiently small $\rho>0:$
\begin{itemize}
\item[(Appr)] There are  functions $A_\rho^+, A_\rho^-:X\rightarrow \mathbb{R}$ such that $A_\rho^\pm\in \mathbb{B}$ and
\begin{itemize}
\item[(i)] $\|A_\rho^\pm\|_\infty\leq 2,$
 $\|A_\rho^\pm\|_\mathbb{B}\leq \rho^{-\tau};$
\item[(ii)] $A_\rho^-\leq 1_{E_\rho}\leq A_\rho^+;$
\item[(iii)] $\mathbb{E}A_\rho^+-\mathbb{E}A_\rho^-\leq  C\fv(\rho)^{1+\eta}.$
\end{itemize}
\end{itemize}

 Let $\{\rho_n\}$ be a decreasing  sequence of positive numbers.
 We say that the sequence  $\{E_{\rho_n}\}$ is a simple  admissible sequence of targets for $(X,\mathcal{B},\mu,f,\mathbb{B})$ if  there exist constants $u>u_0>0$  such that
\begin{equation}\label{rhobound}\tag{Poly}
\exists n_0:\;\forall n\geq n_0,\,\,\,\rho_n \geq n^{-u}, \,\,\, n^{-u}\leq\sigma(\rho_n)\leq n^{-u_0};\end{equation}
and \begin{equation} \tag{Mov} \label{eq.MOV}
 \forall R,\, L\,  \exists C>0:\;
\forall k \in (0,R \ln n),  \;\;\; \mu(E_{\rho_n} \cap E_{\rho_n}^k)\leq C \fv(\rho_n) (\ln n)^{-L}.\end{equation}
\end{Def}

\begin{Def}[Composite admissible targets] \label{def.recurrent.targets}
Let $\{\bfE_\rho\}_{\rho\in \mathbb{R}_+}$ be a decreasing collection of sets in $X\times X$ satisfying the following conditions for some positive constants $\tau, C, \eta,$ and for all
  sufficiently small $\rho>0,$

\begin{itemize}
\item[${\rm(\overline{Appr})}$] There are  functions $\bar{A}_\rho^+, \bar{A}_\rho^-:X\times X\rightarrow \mathbb{R}$ such that $\bar{A}_\rho^\pm\in \mathbb{B}$ and
\begin{itemize}
\item[(i)] $\|\bar{A}_\rho^\pm\|_\infty\leq 2,$ $\|\bar{A}_\rho^\pm\|_\mathbb{B}\leq  \rho^{-\tau};$
\item[(ii)] $\bar{A}_\rho^-\leq 1_{{\bfE}_{\rho}}\leq \bar{A}_\rho^+;$
\item[(iii)] $\DS\int\prod_{i=1}^r\left(\int \brA^-(x_0,x_i) \rd\mu(x_i)\right) \rd\mu(x_0)\geq C^{-1}\bsigma(\rho)^r;$
\item[(iv)] $\forall x,$ $\DS\int\bar{A}_\rho^+(x,y)\rd\mu(y)\leq C\bsigma(\rho)$ and $\forall y,$ $\DS\int\bar{A}_\rho^+(x,y)\rd\mu(x)\leq C \bsigma(\rho);$
\item[(v)] $\DS\int\bar{A}_\rho^+(x,f^kx)\rd\mu(x)\leq C\mu\left({\bfE}_{a_0\rho}^k\right),$ $k\in\mathbb{Z}_+$ for some constant $a_0>0.$
\end{itemize}
\end{itemize}

Let $\{\rho_n\}$ be a decreasing  sequence of positive numbers. The sequence $\{\bfE_{\rho_n}\}$ is said to be composite admissible  if there exist constants $u>u_0>0$ such that
\begin{equation}\label{brhobound}\tag{$\overline{\mathrm{Poly}}$}
\exists n_0:\,\,\,\forall n\geq n_0,\,\,\,\rho_n \geq n^{-u}, \,\, n^{-u}\leq\bsigma(\rho_n)\leq n^{-u_0};
\end{equation}
there is a constant $a>0$ such that for any $k_1<k_2$
 \begin{equation}\label{eq.bSUB} \tag{${\rm \overline{Sub}}$} \bfE_{\rho}^{k_1}\cap\bfE_{\rho}^{k_2}\subset f^{-k_1}\bfE_{a\rho}^{k_2-k_1};\end{equation}
and
   \begin{equation}\label{eq.bMOV} \tag{${\rm \overline{Mov}}$}   \forall L\,\exists n_0: \forall n\geq n_0\,\,\,\text{and}\,\,\,
  k \neq 0,\quad \mu(\bfE^k_{a_0a\rho_n})\leq C (\ln n)^{-L},
 \end{equation}
where the constants $a_0$ and $a$ appear in ${\rm (\overline{Sub})}$ and ${\rm(\overline{Appr})(v)}$ respectively.
\end{Def}

\begin{Rem}
\begin{itemize}
\item[(i)] We introduce ${\rm(\overline{Appr})(iii)}$ and ${\rm(\overline{Appr})(iv)}$ due to the absence of
\bea\label{CTLow}\forall x,\,\,\int1_{{\bfE}_{\rho}}(x,y)\rd\mu(y)\geq C^{-1}\bsigma(\rho)\eea
in certain non-smooth systems. For instance, in dispersing billiards~(refer to Subsection \ref{BidRet}), we observe $$\int1_{{\bfE}_{\rho}}(x,y)\rd\mu(y)=O\left(\rho\sin\rho\cos\varphi\right)\,\,
\text{for}\,\,x=(r,\varphi)\,\,\text{and}\,\,x\notin\partial\mathcal{M},$$
while $$\bsigma(\rho)=O\left(\rho\sin\rho\right).$$
The approach of $x$ to $\partial\mathcal{M}$ results in $\cos\varphi$ approaching $0,$ invalidating \eqref{CTLow}.

\item[(ii)] Using $({\rm GEM})_1$ and ${\rm(\overline{Appr})(i)-(iv)},$ we have
$$\int\bar{A}_\rho^+(x,f^kx)\rd\mu(x)\leq C\left(\int A_\rho^+(x,y)\rd\mu(x)\rd\mu(y)+ \theta^k\left\|\bar{A}_\rho^+\right\|_{\mathbb{B}}\right)\leq C^2\bsigma({\rho})+C\theta^k\rho^{-\tau},$$
$$\mu\left({\bfE}_{\rho}^k\right)\geq C^{-1}\left(\int A_\rho^-(x,y)\rd\mu(x)\rd\mu(y)- \theta^k\left\|\bar{A}_\rho^-\right\|_{\mathbb{B}}\right)\geq C^{-2}\bsigma({\rho})-C^{-1}\theta^k\rho^{-\tau}.$$
\end{itemize}
It follows that for large constant $\brC$ and $k>\left(\ln\frac{\brC C^{-2}+C^2}{\brC C^{-1}+C}+\ln\bsigma(\rho)+\tau\ln\rho\right)/\ln\theta,$
$$\int\bar{A}_\rho^+(x,f^kx)\rd\mu(x)\leq \brC\mu\left({\bfE}_{\rho}^k\right).$$
Therefore, we propose ${\rm(\overline{Appr})(v)}$ to control small $k.$
\end{Rem}

Given a sequence $\{\rho_n\}_{n=1}^\infty$, we recall that $N^n_{\rho_n}(w)$  represents the number of times $k\leq n$ such that $E^k_{\rho_n}$ or $\bfE^k_{\rho_n}$ occurs.
Our goal is to establish conditions for the system $(X,\mathcal{B},\mu,f)$ and the targets $\left\{E_{\rho_{n}}^k\right\}_{(n,k) \in \N^2; 1\leq k \leq 2n}$ or $\left\{\bfE_{\rho_{n}}^k\right\}_{(n,k) \in \N^2; 1\leq k \leq 2n}$ that ensure the validity of the dichotomy stated in Theorem \ref{GMultiBCN} for the number of hits $N^n_{\rho_n}(w)$. To achieve this,  we take
$$ \bS_r=\sum_{j=1}^\infty 2^{rj} \bv_j,$$
where $\bv_j=\fv(\rho_{2^j})^r$ when considering   the targets $\left\{E_{\rho_{n}}^k\right\}_{(n,k) \in \N^2; 1\leq k \leq 2n}$ and ${\bv}_j=\bar\fv(\rho_{2^j})^r$  when considering  the targets  $\left\{\bfE_{\rho_{n}}^k\right\}_{(n,k) \in \N^2; 1\leq k \leq 2n}.$

\begin{Thm} \label{theo.mixing} Assume a system $(X,\mathcal{B},\mu,f,\mathbb{B})$ is generalized {$(2r+1)$}-fold exponentially mixing. Then
\begin{itemize}
\item[(a)] If $\{E_{\rho_n}\}$ is a sequence of simple admissible targets as in Definition \ref{def.targets},  then  the events of the family $\{ E_{\rho_{n}}^k \}_{(n,k) \in \N^2; 1\leq k \leq 2n}$ satisfy $(GM1)_k,$   $(GM2)_k$ and $(GM3)_k$ for $k=1,\cdots,2r+1.$

\item[(b)] If $\{\bfE_{\rho_n}\}$  is a sequence of composite admissible targets as in Definition \ref{def.recurrent.targets}, then  the events of the family $\{ \bfE_{\rho_{n}}^k \}_{(n,k) \in \N^2; 1\leq k \leq 2n}$ satisfy $(GM1)_k,$   $(GM2)_k$ and $(GM3)_k$ for $k=1,\cdots,2r.$
    \end{itemize}
\end{Thm}

Hence, Theorem \ref{GMultiBCN} and Theorem \ref{theo.mixing} imply
\begin{Cor} \label{cor.mixing} If the system $(X,\mathcal{B},\mu,f,\mathbb{B})$ is generalized $(2r+1)$-fold exponentially mixing, and if
 $\{E_{\rho_n}\}$ (or $\{\bfE_{\rho_n}\}$) are as in Definition \ref{def.targets} (or Definition \ref{def.recurrent.targets}), then
 \begin{itemize}
\item[(a)] If $\bS_r<\infty,$   $\mu(H_r)=0.$
\item[(b)] If $\bS_r=\infty,$  $\mu(H_r)>0.$
\end{itemize}
\end{Cor}

The proof of Theorem \ref{theo.mixing} is analogous to that of Theorem 3.7 in \cite{DFL22}, but with small modifications. Therefore we include it in Appendix \ref{prthmmix} for completeness.

\begin{Prop}\label{EssInv}
Let $(X,\mathcal{B},\mu,f)$ be a measure preserving system and $\{E_{\rho_n}\}$ a sequence of measurable subset of $X.$ We denote
$E_{\rho_n}^k=f^{-k}E_{\rho_n}.$ If $E_{\rho_{n_1}}\subset E_{\rho_{n_2}}$  for $n_1\geq n_2,$ and $\lim_{n\rightarrow\infty}\mu(E_{\rho_n})=0,$ then $\mu(H_r\Delta f^{-1}H_r)=0.$
\end{Prop}

A direct corollary of Proposition~\ref{EssInv} is that $\mu(H_r)=0$ or $1$ when $\mu$ is an ergodic measure of $f.$ Hence we get
$\mu(H_r)=1$ under the hypothesis of Theorem~\ref{GMultiBCN}(b).

\begin{proof}[Proof of Proposition~\ref{EssInv}.]
Suppose $Y=\{E_{\rho_n}^1~\text{occur for infinitely many}~n\}.$ Recall that $H_r=\bigcap_{m=1}^\infty\bigcup_{n=m}^\infty\{\omega:N_{\rho_n}^n(\omega)\geq r\}.$ Then
\Bea
H_r\setminus Y&\subseteq&\left\{E_{\rho_n}^k ~\text{occur for at least}~r~\text{times for}~2\leq k\leq n,\,\,i.o.\right\}\\
&\subseteq&\left\{E_{\rho_n}^k ~\text{occur for at least}~r~\text{times for}~2\leq k\leq n+1,\,\,i.o.\right\}\\
&\subseteq&f^{-1}H_r,
\Eea
Because $\mu$ is $f$-invariant and $\mu(Y)=\lim_{n\rightarrow\infty}\mu(E_{\rho_n})=0$ by monotone convergence theorem, we get $\mu(H_r\Delta f^{-1}H_r)=0.$
\end{proof}

\section{ Multiple Logarithm Law for Dispersing Billiard Maps.}
\label{SSMultExpMixBil}
\subsection{Results.}
Let  $\left\{\mathcal{B}_i\right\}_{1\leq i\leq p}$ denote disjoint strictly convex domains on the two-dimensional torus  $\mathcal{T}^2.$ These domains have $C^3$ smooth boundaries with non-vanishing curvature. We consider the motion of a point particle moves freely inside $\mathcal{D}=\mathcal{T}^2\setminus\bigcup_{1\leq i\leq p}\mathcal{B}_i$ at unit speed, following the rule that ``the angle of incidence equals the angle of reflection''.
Suppose $\partial\mathcal{D}$ represents the boundary of the domain $\mathcal{D}.$ The phase space of
the motion  is defined as $\mathcal{M}=\partial\mathcal{D}\times[-\pi/2,\pi/2].$ In this space, we use standard coordinates  $(r,\varphi)$. Here is what these coordinates represent:
 \begin{itemize}
\item[$\bullet$] $r:$ The arc length parameter along the boundary $\partial\mathcal{D};$
\item[$\bullet$] $\varphi:$ The angle between the post-collisional vector $v$ and the inward unit normal vector $n$ at the reflection point $q\in\partial\mathcal{D}.$ Note that $<v,n>=\cos\varphi.$
\end{itemize}

The map $\mathcal{F}:\mathcal{M}\rightarrow \mathcal{M}$  induced by the motion of a billiard particle is commonly referred to as the collision map or billiard map. Importantly, the billiard map preserves a smooth measure given by $d\mu=c_\mu\cos\varphi dr d\varphi,$ where $c_\mu$ represents the normalizing factor. The billiard system $(\mathcal{M},\mathcal{B}(\mathcal{M}),\mu,\mathcal{F})$ is referred to as dispersing billiard or Sinai billiard, as introduced by Sinai~\cite{Sinai70}.

\begin{Thm} \label{BidMultiLog}
Let $(\mathcal{M},\mathcal{B}(\mathcal{M}),\mu,\mathcal{F})$ be a dispersing billiard. Suppose that $d_n^{(r)}(x,y)$  is the $r$-th minimum of $d(x,\mathcal{F}(y)),\cdots, d(x,\mathcal{F}^n (y))$ for $x,\,y\in \mathcal{M},$
 \begin{itemize}
 \item[(a)] For $\mu-$almost every $x,$
	$$\mu\left\{y\in\mathcal{M}:\,\, \limsup_{n\rightarrow \infty} \frac{\left|\ln d_n^{(r)}(x,y)\right|-\frac{1}{2}\ln n}{\ln \ln n}
        = \frac{1}{2r}\right\}=1;$$
 \item[(b)]
$$\mu\left\{x\in\mathcal{M}:\limsup_{n\rightarrow \infty} \frac{\left|\ln d_n^{(r)}(x,x)\right|-\frac{1}{2}\ln n}{\ln \ln n}
        = \frac{1}{2r}\right\}=1.$$
\end{itemize}
\end{Thm}

\subsection{Exponentially mixing property of dispersing billiards.}
We review fundamental properties of dispersing billiards~\cite{ChernovMarkarian06} to establish $(r+1)$-fold exponentially mixing property for a  billiard system $(\mathcal{M},\mathcal{B}(\mathcal{M}),\mu,\mathcal{F})$ and any positive integer $r.$

The billiard map $\mathcal{F}$ is hyperbolic. Namely, there exists a family of $D\mathcal{F}-$invariant unstable cones $C_x^u\subseteq T_x\mathcal{M}$ and a family of $D\mathcal{F}-$invariant stable cones $C_x^s\subseteq T_x\mathcal{M}.$ These cones can be defined such that:
 \begin{itemize}
\item[$\bullet$] For any $(dr,d\varphi)\in C_x^u,$ we have $K_1\leq d\varphi/dr\leq K_2;$
\item[$\bullet$] For any $(dr,d\varphi)\in C_x^s,$ we have $-K_1\leq d\varphi/dr\leq -K_2.$
\end{itemize}
Here, $0<K_1<K_2<\infty$  are constants. We say that a smooth curve $\mathcal{W}\subset\mathcal{M}$ is unstable~(or stable) if at any $x\in\mathcal{W},$ the tangent space belongs to the unstable~(or stable) cones. The hyperbolicity of $\mathcal{F}$ means expansion~(or contraction) on unstable~(or stable) curves. More precisely, for any $x$ lying on an unstable~(or stable) curve in the set $\mathcal{W},$ there exist positive constants $C>0$ and $\Lambda>1$ such that for any positive integer $n$, we have: $$\|D_x\mathcal{F}^n(dx)\|/\|dx\|\geq C\Lambda^n\,\,\left(\text{or}\,\,\|D_x\mathcal{F}^{-n}(dx)\|/\|dx\|\geq C\Lambda^n\right)$$
where $dx$ belongs to the tangent space  $T_x\mathcal{W}.$

The billiard map $\mathcal{F}$ on $\mathcal{M}$ is not smooth. Denote the boundary of the collision space by $\mathcal{S}_0=\partial\mathcal{M}=\{\cos\varphi=0\}.$ The singularity sets for the maps $\mathcal{F}^n$ and  $\mathcal{F}^{-n}$ are given by $\mathcal{S}_n=\bigcup_{i=0}^n\mathcal{F}^{-i}\mathcal{S}_0$ and  $\mathcal{S}_{-n}=\bigcup_{i=0}^n\mathcal{F}^i\mathcal{S}_0,$ respectively. To control the distortions of $\mathcal{F},$  we partition the collision space $\mathcal{M}$ into homogeneity strips $\mathcal{H}_{\pm k},$ parallel to $\partial\mathcal{M}$ and accumulating near it:
\Bea
\mathcal{H}_k&=&\left\{(r,\varphi): \pi/2-k^{-2}<\varphi<\pi/2-(k+1)^{-2}\right\},\\
\mathcal{H}_{-k}&=&\left\{(r,\varphi): -\pi/2+(k+1)^{-2}<\varphi<-\pi/2+k^{-2}\right\},\\
\mathcal{H}_0&=&\left\{(r,\varphi): -\pi/2+k_0^{-2}<\varphi<\pi/2-k_0^{-2}\right\}.
\Eea Define the set of boundaries of homogeneity strips as $\mathbb{S}=\bigcup_{|k|\geq k_0}\partial\mathcal{H}_k.$ We then construct a new collision space $\mathcal{M}_{\mathcal{H}}$ by taking a disjoint union of the closures of the homogeneity strips $\mathcal{H}_0$ and $\mathcal{H}_k$'s, $\left|k\right|\geq k_0.$ Similarly, the map  $\mathcal{F}^n:\mathcal{M}_{\mathcal{H}}\rightarrow\mathcal{M}_{\mathcal{H}}$ is not smooth on the `extended' singularity set
 $\mathcal{S}_n^{\mathbb{H}}=\mathcal{S}_n\bigcup\left(\bigcup_{m=0}^n \mathcal{F}^{-m}(\mathbb{S})\right)$ and the map  $\mathcal{F}^{-n}:\mathcal{M}_{\mathcal{H}}\rightarrow\mathcal{M}_{\mathcal{H}}$ is not smooth on
 $\mathcal{S}_{-n}^{\mathbb{H}}=\mathcal{S}_{-n}\bigcup\left(\bigcup_{m=0}^n \mathcal{F}^{m}(\mathbb{S})\right).$

Let $\xi_n^s(x)$ denote the open connected components of $\mathcal{M}\setminus\mathcal{S}_{n}^{\mathbb{H}}$ containing $x.$
To describe dynamically H\"older continuous functions for billiard maps, we introduce the separation time
 $$s_{+}(x,y)=\min\{n\geq 0: y\notin \xi_n^s(x)\}.$$
That is, the first time when the images $\mathcal{F}^n(x)$ and $\mathcal{F}^n(y)$ lie in different connected components of the new collision space $\mathcal{M}_{\mathcal{H}}.$  Notably, if $x$ and $y$ lie on one unstable curve, then
$$d(x,y)\leq C\Lambda^{-s_{+}(x,y)}.$$ Analogously, we define the past separation time
 $$s_{-}(x,y)=\min\{n\geq 0: y\notin \xi_n^u(x)\},$$
where $\xi_n^u(x)$ denotes the open connected components of $\mathcal{M}\setminus\mathcal{S}_{-n}^{\mathbb{H}}$ containing $x.$

\begin{Def}[Dynamically H\"older continuous functions]\label{dynHol} \cite[Definition 7.26]{ChernovMarkarian06}
A function $A:\mathcal{M}\rightarrow\mathbb{R}$ is said to be dynamically H\"older continuous if there are constants $\theta_A\in(0,1)$ and $K_A>0$ such that for any $x$ and $y$ lying on one unstable curve,
\begin{equation}\label{Holunstb}
|A(x)-A(y)|\leq K_A\theta_A^{s_{+}(x,y)},
\end{equation}
and for any $x$ and $y$ lying on one stable curve,
\begin{equation}\label{Holstb}
|A(x)-A(y)|\leq K_A\theta_A^{s_{-}(x,y)}.
\end{equation}
\end{Def}

\begin{Rem}\label{HolRem}
The space of H\"older continuous functions is contained in the space of dynamically H\"older continuous functions. Namely, suppose  $A$ to be a H\"older continuous function such that
$$|A(x)-A(y)|\leq C_A d(x,y)^{\alpha_A}\quad\forall\,x,\,y\in\mathcal{M}$$
for $C_A>0$ and $0<\alpha_A\leq 1.$ Then $A$ is dynamically H\"older continuous with $K_A=O(C_A)$ and $\theta_A=\Lambda^{-\alpha_A}.$
\end{Rem}

\begin{Rem}
We may define dynamically H\"older continuous functions on $\mathcal{M}^k$ as follows.
A function $A:\mathcal{M}^k\rightarrow\mathbb{R}$ is said to be dynamically H\"older continuous if there are $\theta_A\in(0,1)$ and $K_A>0$ such that for any $x=(x_1,\cdots,x_k),$ $y=(y_1,\cdots,y_k),$
$x_i$ and $y_i(1\leq i\leq k)$ lying on one unstable curve
\begin{equation}\label{MHolunstb}
|A(x)-A(y)|\leq K_A\theta_A^{\min_{1\leq i\leq k}\{s_{+}(x_i,y_i)\}},
\end{equation}
and for any $x=(x_1,\cdots,x_k),$ $y=(y_1,\cdots,y_k),$
$x_i$ and $y_i(1\leq i\leq k)$ lying on one stable curve
\begin{equation}\label{MHolnstb}
|A(x)-A(y)|\leq K_A\theta_A^{\min_{1\leq i\leq k}\{s_{-}(x_i,y_i)\}}.
\end{equation}
Similarly, if we take $A$ to be a H\"older continuous function such that
$$|A(x)-A(y)|\leq C_A d(x,y)^{\alpha_A}\quad\forall\,x,\,y\in\mathcal{M}^k$$
for $C_A>0$ and $0<\alpha_A\leq 1.$ Then $A$ is dynamically H\"older continuous with $\theta_A=\Lambda^{-\alpha_A}$ and some constant $K_A$ depending  on $C_A,$  $k$ and the system.
\end{Rem}

We denote by $\mathcal{DH}(\mathcal{M}^k)$ the space of dynamically H\"older continuous functions on $\mathcal{M}^k$ which are essentially bounded. For $A\in\mathcal{DH}(\mathcal{M}^k),$ we take the norm \bea\label{NomDHCBi}\|A\|_{\mathcal{DH}}=\|A\|_\infty+K_A,\eea
where $\|\cdot\|_\infty$ is the essential supremum norm and $K_A$ is the smallest constant satisfying \eqref{MHolnstb}.

\begin{Thm} \cite[Theorem 7.41]{ChernovMarkarian06}  \label{decaycolmap}
	Let $\mathcal{A}=\{A_i\}_{i=1}^l$, $\mathcal{B}=\{B_j\}_{j=1}^m$ be two sets of dynamically H\"older continuous functions.
	Suppose that there exists constants $\theta_\mathcal{A}\in(0,1),$ $K_\mathcal{A}>0,$ $\|\mathcal{A}\|_\infty>0$ and $\theta_\mathcal{B}\in(0,1),$ $K_\mathcal{B}>0,$ $\|\mathcal{B}\|_\infty>0$ such that for all $i, j,$
	$$\theta_{A_i}\leq\theta_\mathcal{A},\quad K_{A_i}\leq K_\mathcal{A},\quad \|A_i\|_\infty\leq\|\mathcal{A}\|_\infty,$$
and  $$\theta_{B_i}\leq\theta_\mathcal{B},\quad K_{B_i}\leq K_\mathcal{B},\quad \|B_i\|_\infty\leq\|\mathcal{B}\|_\infty.$$

	We consider two products $\tilde{A} = A_0 (A_1 \circ \mathcal{F}^{t_{-1}}) \cdots (A_l \circ \mathcal{F}^{t_{-l}})$ and $\tilde{B} = B_0 (B_1 \circ \mathcal{F}^{t_1}) \cdots (B_m \circ \mathcal{F}^{t_m})$,
	 where $t_{-l} < \cdots < t_{-1} <0<t_1 < \cdots <t_m$. Then there are constants
	 $C_{\mathcal{A}, \mathcal{B}} >0,$ $\theta=\theta_{\mathcal{A}, \mathcal{B}}<1$ such that for all
	 $n>0,$ we have
	\bea\label{MEMBilo} \left| \int \tilde{A} (\tilde{B} \circ \mathcal{F}^n) \rd\mu - \int \tilde{A} \rd \mu \int \tilde{B} \rd\mu  \right| \leq
	C_{\mathcal{A}, \mathcal{B}} \; \theta^n,\eea
where $$C_{\mathcal{A}, \mathcal{B}}=C\|\mathcal{A}\|_\infty^l\|\mathcal{B}\|_\infty^m
\left(\frac{K_\mathcal{A}\|\mathcal{B}\|_\infty}{1-\theta_\mathcal{A}}+
\frac{K_\mathcal{B}\|\mathcal{A}\|_\infty}{1-\theta_\mathcal{B}}+\|\mathcal{A}\|_\infty\|\mathcal{B}\|_\infty\right).$$
\end{Thm}

Note that the constant $\theta=\theta_{\mathcal{A}, \mathcal{B}}$ depends on $\left\{\theta_{A_i}\right\}_{1\leq i\leq l}$ and $\left\{\theta_{B_j}\right\}_{1\leq j\leq m}.$ To capture this dependence, we consider the space of Lipschitz functions as a subspace of dynamically H\"older continuous functions endowed with the norm  $\|\cdot\|_{\mathcal{DH}}$ defined by \eqref{NomDHCBi}. Denoting this space on $\mathcal{M}^k$ as $\mathcal{DH_L}(\mathcal{M}^k),$ we proceed to demonstrate that the system $(\mathcal{M},\mathcal{B}(\mathcal{M}),\mu,\mathcal{F},\mathcal{DH_L}(\mathcal{M}^{r+1}))$ is a generalized $(r+1)$-fold exponentially mixing system,  in accordance with  Definition~\ref{def.mixing}.

We begin  the proof by noting that  $\rm{(Prod)}$  is a direct consequence of \eqref{NomDHCBi} and we can take $L=\Lambda$ for  $\rm{(Gr)}$ by Remark \ref{HolRem}.  Subsequently, invoking \eqref{MEMBilo} yields the multiple exponentially mixing property of $\mathcal{F}.$ Specifically, there exist constants  $C>0$ and $\theta<1$ such that for any  $A_0, A_1,\dots, A_r\in\mathcal{DH_L}(\mathcal{M})$
and any $r$ tuple $k_1<k_2<\dots<k_r,$
 \begin{equation}\label{MEMBil}
 \left| \int \prod_{j=0}^{r} \left(A_j \circ \mathcal{F}^{k_j} \right) \rd \mu  -
 \prod_{j=0}^{r}\int A_j \rd\mu  \right| \leq C \theta^n \prod_{j=0}^{r} \|A_j\|_{\mathcal{DH}},
 \end{equation}
 where $\DS n=\min_j (k_j-k_{j-1})$ with $k_0=0.$

For the generalized $(r+1)$-fold exponentially mixing property, we consider a larger class of functions. Specifically, we require the existence of constants $\brC>0$ and $\brtheta<1$ such that for any Lipschitz function $B\in{\mathcal{DH_L}(\mathcal{M}^{r+1})},$ the following inequality is satisfied:
 \begin{equation}\label{MEMBilF}
\left|\int B(x_0,\mathcal{F}^{k_1}x_0,\cdots,\mathcal{F}^{k_r}x_0) \rd\mu(x_0)-\int B(x_0,\cdots,x_r)\rd\mu(x_0)\cdots \rd\mu(x_r)\right|\leq \brC\brtheta^n\left\|B\right\|_{\mathcal{DH}},
 \end{equation}
 where $\DS n=\min_j (k_j-k_{j-1})$ with $k_0=0.$

Using the same argument as presented in Appendix A of \cite{DFL22}, we can establish the equivalence between equations \eqref{MEMBil} and \eqref{MEMBilF}. Namely, it demonstrates that if either \eqref{MEMBil} or \eqref{MEMBilF} is valid for $C^s$ functions with endowed with $\|\cdot\|_{C^s}$, then this holds for Lipschitz functions with norm $\|\cdot\|_{\mathcal{DH}},$ and the converse is also true. Furthermore,  \eqref{MEMBilF} arises from \eqref{MEMBil}  through the decomposition of $C^s$ functions relative to a basis within the Sobolev space, utilizing the norm  $\|\cdot\|_{C^s}.$ This completes the verification of Definition~\ref{def.mixing} for the system $(\mathcal{M},\mathcal{B}(\mathcal{M}),\mu,\mathcal{F},\mathcal{DH_L}(\mathcal{M}^{r+1})).$

\subsection{Proof of Theorem~\ref{BidMultiLog}(a).}
Suppose $x\in\mathcal{M}\setminus\bigcup_{n\in\mathbb{N}}\left(\mathcal{S}_{n}^\mathbb{H}\cup\mathcal{S}_{-n}^\mathbb{H}\right).$
We set $$E_\rho=B(x, \rho)=\{y:d(x,y)\leq\rho\}\,\,\,\text{for}\,\,\,\rho<d(x,\mathcal{S}_0^\mathbb{H}),$$
$$E_\rho^k=\mathcal{F}^{-k}E_\rho\quad\text{and}\quad\rho_n=n^{-1/2}(\ln n)^{-\delta},\quad\delta>0.$$
We now show that the sequence $\{E_{\rho_n}\}$ constitutes simple admissible targets as Definition~\ref{def.targets}. To do this, we  break down the process into several steps:
 \begin{itemize}
\item[$\bullet$] Direct calculation gives that $$\sigma(\rho)=\mu(E_\rho)=O(\rho\sin\rho\cos\varphi).$$
\item[$\bullet$] We select functions $A_\rho^+,\,A_\rho^-\in\mathcal{DH_L}(\mathcal{M})$ such that $$1_{E_{\rho-\rho^s}}\leq A_\rho^-\leq 1_{E_\rho}\leq A_\rho^+\leq 1_{E_{\rho+\rho^s}}\,\, \text{for some}\,\, s>1\,\, \text{and}\,\, \|A_\rho^\pm\|_\mathbb{\mathcal{DH}}\leq \rho^{-2s}.$$
\item[$\bullet$] Then $$\mathbb{E}A_\rho^+-\mathbb{E}A_\rho^-\leq  C\fv(\rho)^{(1+s)/2},$$ which verifies  $\rm{(Appr)}.$
\item[$\bullet$] Finally we have the following proposition for $\rm{(Mov)}.$
\end{itemize}

\begin{Prop}
For almost every $x$ and  each
$A, K>0$, there exists $n_0=n_0(x)$ such that for all $n\geq n_0$ and all $k\leq K\ln n$ we have
\begin{equation}
\label{EqWNRBid}
 \mu\left(B(x, \rho_n) \cap \mathcal{F}^{-k} B(x, \rho_n)\right)\leq  \mu \left(B(x, \rho_n)\right) (\ln n)^{-A}.
\end{equation}

\end{Prop}

We begin by proving \eqref{EqWNRBid} for $k\geq \epsilon\ln n,$ $\epsilon>0.$

We call an unstable curve  $W\subset\mathcal{M}$ weakly homogeneous if $W$ belongs to a single homogeneous strip $\mathbb{H}_k.$ We set a small disc $U$ centered at $x$ that is foliated by weakly homogeneous unstable curves $W^u_x(\rho_n) $  with length $\rho_n$, such that $  B(x,\rho_n)\subset U \subset B(x,C\rho_n)\subset\mathbb{H}_k$ for some constant $C>0.$

For $y\in U$ and $n\geq 0$, we denote by $W^u_x(y,\rho_n)$ the weakly homogeneous unstable curve containing the point $y$, $W^u_{x,k}(y,\rho_n)$ the connected component of $\mathcal{F}^k \left(W^u_{x} (y,\rho_n)\right)$ containing the point $\mathcal{F}^k (y),$  Then we obtain
\begin{equation}
\label{BillardIntEst}
\mu\left(B(x, \rho_n) \cap \mathcal{F}^{-k} B(x, \rho_n)\right) \leq   \mu (U \cap \mathcal{F}^{-k} U)\leq C\int \left|W^u_{k,x}(y,\rho_n) \cap U\right| \rd y,
\end{equation}
for some constant $C>0,$ where $|\cdot|$ denotes the length of the curve.
For the purpose of  calculating the intersection of $U$ and all connected components of one homogeneous unstable curve stretching by  $\mathcal{F}^k$, we require the following Growth Lemma to  characterize the size of $W^u_{x,k}(y,\rho_n).$

\begin{Lem}\cite[Theorem 5.52]{ChernovMarkarian06}
There are constants $\hat \Lambda >1$, $\delta \in (0,1)$, $C_1, C_2>0$ such that for all $k>0$, $\varepsilon>0,$
\begin{equation}
\label{EqGL}	
	 m_W(r_k(y)<\varepsilon) \leq C_1 (\delta\hat \Lambda)^k m_W(r_0(y) < \varepsilon / \hat{\Lambda}^k )  + C_2 \varepsilon m_W(W),
\end{equation}	
where $m_W$ is the Lebesgue measure on $W,$ and  $r_k(y)$ is the distance from the point $\mathcal{F}^k(y)$ to the nearest endpoint of $W^u_{x,k}(y,\rho_n)$.
\end{Lem}

Due to (\ref{EqGL}),
$$m_W\left(\left|W_{k,x}^u(y,\rho_n)\right|<\delta^{-k/2} \rho_n \right)\leq m_W\left(r_k(y) < \delta^{-k/2} \rho_n\right) \leq C_1 \delta^{k/2}\rho_n +C_2 \delta^{-k/2} \rho_n^2,$$
which implies ``most'' of such components are long.
For the short components,
\begin{equation}
\label{BillardShortUns}
\int_{\left|W^u_{k,x}(y,\rho_n)\right| <\delta^{-k/2} \rho_n} \left|W^u_{k,x}(y,\rho_n) \cap U\right| {\rm d} y\leq  (C_1 \delta^{k/2}\rho_n +C_2 \delta^{-k/2} \rho_n^2)\rho_n.
\end{equation}

Without loss of generality, we suppose that  $\diam\mathcal{M}=1.$
In the case of the ``long'' components, that is  $\left|W_{k,x}^u(y,\rho_n)\right|\geq\delta^{-k/2}\rho_n,$  we have the following two cases.

(i) If $\left|W_{k,x}^u(y,\rho_n)\right|\leq 1,$
then $W_{k,x}^u(y,\rho_n) $ intersects $U$ for at most one time. Hence $\left|W^u_{k,x}(y,\rho_n) \cap U\right|\leq \rho_n,$
\begin{equation}
\label{BillardLongUns1}
\int_{\left|W^u_{k,x}(y,\rho_n)\right| \geq \delta^{-k/2} \rho_n} \left|W^u_{k,x}(y,\rho_n) \cap U\right| {\rm d} y\leq C\mu(U)\delta^{k/2} \leq C\mu(B(x,\rho_n))n^{\epsilon\ln\delta/2},
\end{equation}
for $k\geq \epsilon\ln n.$

(ii) If $\left|W_{k,x}^u(y,\rho_n)\right|>1,$ then $W_{k,x}^u(y,\rho_n) $ may intersect $U$ for more than one time,
we get
$$\left|W_{k,x}^u(y,\rho_n)\cap U\right|/\left|W_{k,x}^u(y,\rho_n)\right|\leq 2\rho_n.$$
Hence
\begin{equation}
\label{BillardLongUns2}
\int_{\left|W^u_{k,x}(y,\rho_n)\right| \geq \delta^{-k/2} \rho_n} \left|W^u_{k,x}(y,\rho_n) \cap U\right| \rd y\leq C\mu(U)\rho_n \leq C\mu(B(x, \rho_n))^{\frac32}.
\end{equation}
Combining (\ref{BillardIntEst}), (\ref{BillardShortUns}), (\ref{BillardLongUns1}) and (\ref{BillardLongUns2}), we obtain
$$\mu\left(B(x, \rho_n) \cap \mathcal{F}^{-k} B(x, \rho_n)\right)  \leq \mu\left(B(x, \rho_n)\right)^{1+\eta}
$$
for some $\eta >0$ and $k\geq \epsilon\ln n.$

Next, we consider $B(x, \rho) \cap \mathcal{F}^{-k} B(x, \rho)$ for $k<\epsilon\ln n.$

\begin{Def}\label{Diop}
Let $f$ be a transformation on $X$ preserving the measure $\mu.$ We call $x\in X$ {\it Diophantine} if there exists $\epsilon=\epsilon(x)>0$ and $\rho_0=\rho_0(x)>0$ such that for all $\rho\leq \rho_0$ and all $0<k\leq \epsilon|\ln \rho|,$
$f^{-k} B(x, \rho) \cap B(x, \rho)
= \emptyset.$
\end{Def}

We now  prove that for the dispersing billiard map $\mathcal{F},$ $\mu$-almost every point is  Diophantine.
\begin{Prop}\cite[Theorem 1]{STV02}\label{prop1}
  Let $f: X\rightarrow X$ be a measurable transformation preserving an ergodic probability measure
  $\mu$ and $\tau(A) = \inf \{ n\geq1, f^n A \cap A \neq \emptyset \}$.
  If $\mathcal{P}$ is a finite or countable measurable partition with $h_\mu (f,\mathcal{P})>0$ and $P_n(x)$ is the cylinder of length $n$ containing $x$,
  then almost surely
  $$\liminf_{n\rightarrow \infty} \frac{\tau(P_n(x))}{n} \geq 1.$$
\end{Prop}

Recall that the collision space $\mathcal{M}$ is divided into homogeneity strips $\mathcal{H}_{\pm k},$ $\mathcal{M}_{\mathcal{H}}$ is a disjoint union of the closures of the $\mathcal{H}_k$'s and  $\partial\mathcal{M}_{\mathcal{H}}=\mathbb{S}.$ Then
 $$ \sum_{n=1}^{\infty}  \mu\left\{ x: d\left(\mathcal{F}^n (x),\mathbb{S}\right)< e^{-an} \right\} < \sum_{n=1}^\infty C e^{-an/2} < \infty,\quad a>0.$$

By Borel-Cantelli Lemma, for a.e. $x\in \mathcal{M}$, $\exists\, n_0 (x)$, $n> n_0(x)$,
\bea\label{biddis}d\left(\mathcal{F}^n (x),\mathbb{S}\right) >e^{-an}.\eea
Fix $x\in\mathcal{M}\setminus\mathbb{S}$ and satisfies \eqref{biddis}, $\exists a_0>a$, s.t.
\bea\label{BidBouDis}  d\left(\mathcal{F}^n (x),\mathbb{S}\right)> e^{-a_0n}, \,\,\forall n\in \mathbb{N}. \eea

We know from $\S$5.5 in  \cite{ChernovMarkarian06} that the derivative of dispersing billiard map is uniformly bounded in each homogeneity strips $\mathbb{H}_k,$ i.e.
\begin{equation}
\label{BidExd}
\left\|D_z\mathcal{F}\right\|\leq c|k|^2,\,\, z\in\mathbb{H}_k
\end{equation}
for some constant $c>0.$ Denote $\beta=\max\{\ln(2c),\,2a_0\}.$ Then we can take $\mathbb{H}=\{\mathcal{H}_{\pm k}\}_{k\geq k_0}\bigcup\{\mathcal{H}_0\}$ as a partition of $\mathcal{M}$ and
\begin{equation}
\label{BidPatN}
\mathbb{H}^n =\bigvee_{j=0}^{n-1} \mathcal{F}^{-n}\mathbb{H}=\left\{\mathcal{H}_{i_0,\ldots,i_{n-1}}:\,\mathcal{H}_{i_0,\ldots,i_{n-1}} =\mathcal{H}_{i_0} \cap \mathcal{F}^{-1} \mathcal{H}_{i_1} \cap \cdots \cap \mathcal{F}^{-(n-1)} \mathcal{H}_{i_{n-1}},\, \mathcal{H}_{i_j}\in\mathbb{H}\right\}\end{equation}
the dynamical partition for any integer $n.$
We note that
 \bea\label{BidBouInd} i_k\leq \sqrt 2e^{ka_0/2}\eea
by \eqref{BidBouDis} and the definition of $\mathcal{H}_k.$

We process to  prove that $\beta(x,e^{-\beta n} ) \subset\mathbb{H}^{n} (x)$ by induction, where $\mathbb{H}^{n} (x)$ represents the member of  $\mathbb{H}^n$ containing $x.$ First, we observe that
 $$d(x,y) < e^{-\beta n} < e^{-a_0} < d(x,\mathbb{S}),$$
then $x$ and $y$ in the same homogeneity strip $\mathcal{H}_{i_0}.$

Using \eqref{BidExd} and \eqref{BidBouInd}, we have $$\left\|D_z\mathcal{F}\right\|\leq c|i_0|^2\leq  2ce^{a_0},\quad z\in\mathbb{H}_{i_0}.$$
It follows that $$d\left(\mathcal{F}(x),\mathcal{F}(y)\right)\leq e^{-\beta n}2ce^{a_0}\leq e^{-2a_0},$$
and $\mathcal{F}(x)$ and $\mathcal{F}(y)$ are in the same homogeneity strip $\mathcal{H}_{i_1}.$

By induction,
$$  d\left(\mathcal{F}^k (x), \mathcal{F}^k (y)\right) \leq e^{-\beta n}2ce^{a_0 k}\leq e^{-a_0 k} \leq d\left(\mathcal{F}^k (x), \mathbb{S}\right),\quad k\leq n,$$
which implies that $B(x,e^{-\beta n})\subseteq\mathbb{H}^{n} (x).$

By Theorem 3.42 in~\cite{ChernovMarkarian06}, $$h_\mu(\mathcal{F})=\int_{\mathcal{M}}\lambda_x^+\rd\mu(x),$$
where $\lambda_x^+$ is the positive Lyapunov exponent of the map $\mathcal{F}$ at the point $x\in\mathcal{M}.$ We observe that for almost every $x,$ $\lambda_x^+\geq\log\Lambda,$ since $\|D_x\mathcal{F}^n(dx)\|/\|dx\|\geq C\Lambda^n$ for any $x\in W$ on an unstable curve $W$ and $dx\in T_x W.$ Consequently, we have
$$h_\mu (\mathcal{F})\geq\log\Lambda>0.$$ Then $h_\mu \left(\mathcal{F},\mathbb{H}\right)>0$ follows from the fact that the members of
 $\bigvee_{m=-n}^{n}\mathcal{F}^m\mathbb{H}$ shrink to points when $n\rightarrow\infty,$ which implies $\bigvee_{m=-\infty}^{+\infty}\mathcal{F}^m\mathbb{H}=\mathcal{B}(\mathcal{M}).$

Therefore, Proposition \ref{prop1} tells us that for almost every $x\in \mathcal{M},$
$$\liminf_{\rho\rightarrow \infty} \frac{\tau\left(B(x,\rho)\right) }{-\log \rho}  = \liminf_{n\rightarrow \infty} \frac{\tau\left(B(x,e^{-\beta n})\right) }{\beta n} \geq \liminf_{n\rightarrow \infty} \frac{\tau\left(\mathbb{H}^n(x)\right)}{\beta n}  \geq \frac{1}{\beta}.$$

Therefore, $B(x,\rho) \cap \mathcal{F}^{-k} B(x,\rho) = \emptyset$ whenever $k< -\frac{1}{2 \beta} \log \rho$ and $\rho$ small, which completes the proof of Theorem~\ref{BidMultiLog}(a) by Corollary~\ref{cor.mixing} and Proposition~\ref{EssInv}.

\subsection{Proof of Theorem \ref{BidMultiLog}(b)}\label{BidRet}
Let $\mathcal{X}=\mathcal{M}\setminus\bigcup_{n\in\mathbb{N}}\left(\mathcal{S}_{n}^\mathbb{H}\cup\mathcal{S}_{-n}^\mathbb{H}\right).$
We take $$\bfE_{\rho}=\left\{(x,y)\in\mathcal{X}\times\mathcal{X}:d(x,y)\leq\min\left\{\rho,d(x,\mathcal{S}_0^\mathbb{H})\right\}\right\} ,\quad \bfE_\rho^k=\left\{x:\left(x,\mathcal{F}^k(x)\right)\in\bfE_{\rho}\right\},$$
$$\bar{\fv}(\rho)=(\mu\times\mu)(\bfE_\rho)\quad\text{and}\quad\rho_n=n^{-1/2}(\ln n)^{-\delta},\,\,\delta>0.$$
We now demonstrate that the sequence $\{\bfE_{\rho_n}\}$ constitutes composite admissible targets as Definition~\ref{def.recurrent.targets}. The proof will proceed as follows:
 \begin{itemize}
\item[$\bullet$] Observe that $\bsigma(\rho)=O\left(\rho\sin\rho\right)$ and
 $\DS\int\prod_{i=1}^r\left(\int 1_{\bfE_\rho}(x_0,x_i) \rd\mu(x_i)\right) \rd\mu(x_0)=O\left(\rho^r(\sin\rho)^r\right).$
\item[$\bullet$] We select functions $\bA_\rho^-,\,\bA_\rho^+\in\mathcal{DH_L}(\mathcal{M}^2)$ such that $$1_{\bfE_{\rho-\rho^s}}\leq \bA_\rho^-\leq 1_{\bfE_\rho}\leq \bA_\rho^+\leq 1_{\bfE_{\rho+\rho^s}}\,\, \text{for some}\,\, s>1\,\, \text{and}\,\, \|\bA_\rho^\pm\|_\mathbb{\mathcal{DH}}\leq \rho^{-2s}.$$
\item[$\bullet$] Then $$\mathbb{E}\bA_\rho^+-\mathbb{E}\bA_\rho^-\leq  C\bar{\fv}(\rho)^{(1+s)/2},$$ which verifies ${\rm(\overline{Appr})(i)-(iii)}.$
\item[$\bullet$]     Note that
$$\int\bar{A}_\rho^+(x,y)\rd\mu(y)\leq\mu (B(x,{\rho+\rho^s}))\leq C\rho\sin\rho$$
and a similar inequality holds for $\DS\int\bar{A}_\rho^+(x,y)d\mu(x).$ Then ${\rm(\overline{Appr})(iv)-(v)}$ follow by $$\int\bar{A}_\rho^+\left(x,\mathcal{F}^k(x)\right)\rd\mu(x)\leq\mu\left(\bfE_{\rho+\rho^s}^k\right)\leq\mu\left(\bfE_{2\rho}^k\right).$$
\item[$\bullet$] Next $({\rm \overline{Sub}})$ is valid because $$\bfE_{\rho}^{k_1}\cap\bfE_{\rho}^{k_2}\subset \mathcal{F}^{-k_1}\bfE_{2\rho}^{k_2-k_1}.$$
\item[$\bullet$] Finally,   the following Proposition shows $({\rm \overline{Mov}})$.
\end{itemize}

\begin{Prop}\label{SlowREc-Bid}
For each $A>0,$ there exists $\rho_0>0$ such that for all $\rho<\rho_0$ and all $k\in \mathbb{Z}_+$ we have
\begin{equation}
\label{EqWNRSBi}
 \mu\left(\left\{x:d\left(x,\mathcal{F}^k(x)\right)<\rho\right\}\right)\leq C|\ln \rho|^{-A}.
\end{equation}
\end{Prop}

\begin{proof} We prove \eqref{EqWNRSBi} for sufficiently large $A,$ then  \eqref{EqWNRSBi} holds for each $A>0.$
Take $B=A^2$. If $k\geq B\ln|\ln\rho|,$ we take $\hat\rho=|\ln \rho|^{-A}.$ By $\rm{(GEM)_1},$ we obtain that
\begin{multline}
\label{eq.highDBid}
\mu \left(\left\{x: d\left(x,\mathcal{F}^k(x)\right)\leq \rho\right\}\right)\leq\mu\left(\left\{x: d\left(x,\mathcal{F}^k(x)\right)\leq \hat\rho\right\}\right)\\
\leq \int\bar{A}^+_{\hat{\rho}}\left(x,\mathcal{F}^k(x)\right)\rd\mu(x) \leq C\left((\hat\rho+\hat{\rho}^s)^2+\hat\rho^{-\tau} \theta^k\right)\leq |\ln \rho|^{-2A},
\end{multline}
provided that $A\geq-2(s+1)/\ln\theta$ and $\rho$ is sufficiently small.

Now fix any $1\leq k\leq B\ln|\ln\rho|$.  Assume that $x$ satisfies $d\left(x,\mathcal{F}^k(x)\right)\leq \rho.$ We denote
$$\overline{\mathcal{H}}_n=\bigcup_{i\leq n}\mathcal{H}_i,\quad \overline{\mathcal{H}}_{n,l}=\bigcap_{j=1}^l\mathcal{F}^{-j}\overline{\mathcal{H}}_n,\quad n,\,l\in\mathbb{Z}_+.$$
Then
$$\mu\left(\overline{\mathcal{H}}_{n,l}^c\right)\leq\sum_{j=1}^l\mu\left(\overline{\mathcal{H}}_{n}^c\right)\leq ln^{-2}.$$
By \eqref{BidExd}, $$\left\|D_z\mathcal{F}\right\|\leq cn^2,\quad z\in\overline{\mathcal{H}}_{n}$$
for some $c>0.$
Then for any $x\in\overline{\mathcal{H}}_{n,l},$ $l\geq jk$ we have that
$$d\left(\mathcal{F}^{(j-1)k}(x), \mathcal{F}^{jk} (x)\right)\leq (cn^2)^{(j-1)k} \rho.$$
If we take $n=\lfloor|\ln\rho|^{2B}\rfloor,$ $l=\lfloor8B\ln|\ln\rho|/k\rfloor,$ $L=\lceil4B\ln|\ln\rho|/k\rceil,$ we find that
$$d\left(x,\mathcal{F}^{Lk} (x)\right)\leq \sum_{j=1}^L d\left(\mathcal{F}^{(j-1)k}(x),\mathcal{F}^{jk}(x)\right)
\leq L (cn^2)^{4B\ln|\ln\rho|}\rho\leq \sqrt{\rho},$$
 provided that $\rho$ is sufficiently small. Since $Lk\geq B\ln|\ln\rho|,$ we apply \eqref{eq.highDBid} and obtain
$$\mu \left(\left\{x: d\left(x, \mathcal{F}^k (x)\right)\leq \rho\right\}\right)
\leq \mu(\overline{\mathcal{H}}_{n,l}^c)+\mu \left(\left\{x\in\overline{\mathcal{H}}_{n,l}: d(x, \mathcal{F}^k (x))\leq \rho\right\}\right)$$
$$\leq l n^{-2}+\mu \left(\left\{x: d(x, \mathcal{F}^{Lk} (x))\leq \sqrt\rho\right\}\right)\leq C|\ln \rho|^{-A}.
$$
\end{proof}

We have obtained that $\left(\mathcal{M},\mathcal{B}(\mathcal{M}),\mu,\mathcal{F},\mathcal{DH_L}(\mathcal{M}^{r+1})\right)$ is generalized $(r+1)$-fold exponentially mixing
as Definition~\ref{def.mixing} and $\{\bfE_{\rho_n}\}$ is composite admissible as Definition~\ref{def.recurrent.targets}. By Corollary~\ref{cor.mixing},
$$\limsup_{n\rightarrow \infty} \frac{\left|\ln d_n^{(r)}(x,x)\right|-\frac{1}{2}\ln n}{\ln \ln n}\leq \frac{1}{2r}\,\,\text{for}\,\, a.e.-x$$
and
$$\limsup_{n\rightarrow \infty} \frac{\left|\ln d_n^{(r)}(x,x)\right|-\frac{1}{2}\ln n}{\ln \ln n}\geq \frac{1}{2r}\,\,\text{on a positive measure set}.$$

If we denote
$$R_{\rm{Bid}}=\left\{x\in\mathcal{M}:\limsup_{n\rightarrow \infty} \frac{\left|\ln d_n^{(r)}(x,x)\right|-\frac{1}{2}\ln n}{\ln \ln n}
        = \frac{1}{2r}\right\},$$
then $\mu\left(R_{\rm{Bid}}\right)>0.$ We prove that $\mu\left(R_{\rm{Bid}}\right)=1.$

Suppose that $x\in R_{\rm{Bid}}\setminus\bigcup_{n\in\mathbb{N}}\left(\mathcal{S}_{n}^\mathbb{H}\cup\mathcal{S}_{-n}^\mathbb{H}\right).$ Then  $x$ is contained within the homogeneity strips $\mathcal{H}_m$ for some $m\in\mathbb{Z}_+.$ Given any $\epsilon>0,$ there exists an increasing sequence $\{n_k\}_{k=1}^\infty$ satisfying
$$d_{n_k}^{(r)}(x,x)<\frac{1}{n_k^{\frac12}\left(\ln\ln n_k\right)^{\frac{1}{2r}-\epsilon}},\quad\forall k\in\mathbb{Z}_+$$
and $\left\|D_z\mathcal{F}\right\|\leq c|m|^2$ for $z\in B(x,\rho)$ and $\rho$ sufficiently small by \eqref{BidExd}. It follows that
$$d_{n_k}^{(r)}\left(\mathcal{F}(x),\mathcal{F}(x)\right)<\frac{c|m|^2}{n_k^{\frac12}\left(\ln\ln n_k\right)^{\frac{1}{2r}-\epsilon}},\quad\forall k\in\mathbb{Z}_+$$
and
$$\limsup_{n\rightarrow \infty} \frac{\left|\ln d_n^{(r)}\left(\mathcal{F}(x),\mathcal{F}(x)\right)\right|-\frac{1}{2}\ln n}{\ln \ln n}\geq \frac{1}{2r},\quad\forall k\in\mathbb{Z}_+.$$
Thus $R_{\rm{Bid}}\setminus\bigcup_{n\in\mathbb{N}}\left(\mathcal{S}_{n}^\mathbb{H}\cup\mathcal{S}_{-n}^\mathbb{H}\right)\subseteq\mathcal{F}^{-1}R_{\rm{Bid}}.$ In addition since $\mu\left(\mathcal{S}_n^\mathbb{H}\right)=0$ for $n\in\mathbb{Z},$ we obtain $$\mu\left(R_{\rm{Bid}}\Delta\mathcal{F}^{-1}R_{\rm{Bid}}\right)=0.$$
This completes the proof of Theorem~\ref{BidMultiLog}(b), by the ergodicity of the measure $\mu.$

\section{ Multiple Logarithm Law for Piecewise Expanding Maps.}
\label{SSMultExpMixPem}
\subsection{Results.}
Piecewise expanding maps play a crucial role in understanding chaotic phenomena. These maps arise both as direct models for chaotic systems and  integral components  in the analysis of other mathematical models. Notable examples include tent map, Lorentz-like map and Gauss map. Our goal in this section is to establish the multiple Logarithm Law for hitting times and recurrence of piecewise expanding maps under some regularity conditions.

Let $T$ be a  piecewise expanding map on the interval $I=[0,1],$ endowed with a set of finite and countably many singular points
$$\mathcal{S}=\left\{a_j\in I:\,\,1=a_0>a_1>\cdots>a_N=0\right\},$$
where $N\leq\infty$ could be a finite number or $\infty.$ For each $0\leq j<N,$ the restriction of the map $T|_{\Delta_j}:\Delta_j=(a_{j+1},a_j)\rightarrow T(\Delta_j)$ is of class $C^3$ and strictly monotonic. We call $T$  an expanding map if there exist constants  $C>0$ and $\lambda>1$ such that whenever $(T^n)'(x)$ is defined, it satisfies $\left|(T^n)'(x)\right|\geq C\lambda^n,$ $\forall\,n\in\mathbb{Z}_+.$ We further assume that $T$ is surjective and topologically exact, meaning that for any subinterval $J\in I\setminus\mathcal{S},$ there exists $n(J)\geq 1$ such that
$T^n(J)=I.$

Suppose that $m$ is the Lebesgue measure on ${I}.$ We  assume that
\begin{itemize}
\item[(PE1)] There exists $C>0$ such that for every $i$ and $x\in \Delta_i$, $\frac{|T''(x)|}{|T'(x)|^2}<C.$
\item[(PE2)] There exists $\delta>0$ such that for every $i$ and $x\in \Delta_i$, $m\left(T(\Delta_i)\right)>\delta.$
\end{itemize}

Under the assumptions  (PE1) and (PE2), $T$ admits a unique absolutely continuous ergodic measure $\mu$ such that $\rd\mu=\gamma \rd m$ by \cite[Proposition 3.18]{Viana97}, where $\gamma$ has bounded variation. We additionally introduce assumptions for the Lyapunov exponent of $\mu$ and boundary respectively:
\begin{itemize}
\item[(PE3)] $\DS\int\log|T'(x)|\rd\mu(x)=\Lambda_0<\infty.$
\item[(PE4)] Denote by $\Delta = \{\Delta_i\}_{i=1}^\infty.$
Then the boundary $\mathcal{S}$ of $\Delta$ satisfies
 $$ \mu\left( \{ x: d(x,\mathcal{S}) < \varepsilon \}\right) < C\varepsilon^\gamma, $$
 for some $C>0,\,\gamma>0.$
\end{itemize}

\begin{Thm} \label{PemMultiLogHit}
Let $T:{I}\rightarrow {I}$ be a piecewise expanding map. Suppose that $d_n^{(r)}(x,y)$  be the $r$-th minimum of $d(x,T(y)),\cdots, d(x,T^n (y))$ for $x,\,y\in I.$
 If {\rm (PE1)--(PE4)} are satisfied, for $\mu$-almost every $x,$
 	\bea\label{PemHit}\mu\left\{y\in I:\,\,\limsup_{n\rightarrow \infty} \frac{\left|\ln d_n^{(r)}(x,y)\right|-\ln n}{\ln \ln n}
        = \frac{1}{r}\right\}=1.\eea
In particular, \eqref{PemHit} holds for tent map, Lorentz-like map and Gauss map.
\end{Thm}

 \begin{Rem}
 \begin{itemize}
  \item[(i)]  Assumption {\rm (PE1)} is  known as the bounded distortion property. Specifically,  for each $y, z$ in the same continuity domain of $T^n,$
\bea\label{pebdt}\left| \frac{(T^n)'(y)}{(T^n)'(z)} \right|< e^ C.\eea
   \item[(ii)] We denote  the inverse Jacobian of $T$ to be $h_j(x)=|T'(x)|^{-1},\,\,x\in\Delta_j$ for each $0\leq j<N.$ Then assumption {\rm (PE2)} can be replaced by the condition: $$\text{Var}(h_j)/\sup(h_j)<\infty\quad\text{and}\quad\Sigma_j\sup(h_j)<\infty$$
        where   $\text{Var}(\psi)$ is the variation of $h_j.$
   \item[(iii)] By assumption {\rm (PE3)}, we obtain for $\mu$-almost every $x,$
\bea\label{RELY}\lim_{n\rightarrow\infty}\frac{1}{n}\log|(T^n)'(x)|=\int\log|T'(x)|\rd\mu(x)=\Lambda_0.\eea
 \item[(iv)] Assumption {\rm (PE4)} is not very restrictive. In particular, it holds when $\mathcal{S}$ is a finite set and for  Gauss map.
\end{itemize}
\end{Rem}

To establish multiple Logarithm Law for recurrence, we introduce assumptions pertaining to the singularities of piecewise expanding maps, as proposed in \cite{ChenZhang19}:
\begin{itemize}
\item[$\rm\overline{(PE1)}$] There exist $\sigma_i^\pm\geq 0$ such that
\bea\label{bddisbdy}\limsup_{x\rightarrow a_i^\pm}\frac{|T''(x)|}{|T'(x)|^{\sigma_i^\pm+2}}<\infty,\eea
$$\limsup_{x\rightarrow a_i^\pm}\frac{|T^{(3)}(x)|}{|T'(x)|^{\sigma_i^\pm+2}}<\infty.$$
\item[$\rm\overline{(PE2)}$] There exists $0\leq\alpha_i<1$ such that
$$0<\liminf_{x\rightarrow a_i^\pm}\left|(x-a_i)^{\alpha_i}T'(x)\right|<\limsup_{x\rightarrow a_i^\pm}\left|(x-a_i)^{\alpha_i}T'(x)\right|<\infty.$$
\item[$\rm\overline{(PE3)}$] There are only finitely many singularities $a_i$'s such that  $\sigma_i^\pm>0,$ and for such $a_i$ which are not fixed points, $\alpha_i^\pm\sigma_i^\pm<1.$
\end{itemize}

Given the assumptions $\rm\overline{(PE1)}$--$\rm\overline{(PE3)}$, the piecewise expanding map $T$ is associated with a unique ergodic measure $\nu,$  which is absolutely continuous with respect to the Lebesgue measure $m$. This measure can be expressed as $\rd\nu=\kappa\rd m,$ where $\kappa$ is positive and continuous  with the exception of a countable set~\cite[Theorem 1.1]{ChenZhang19}.

\begin{Thm}\label{PemMLgRec}
Let $T:{I}\rightarrow {I}$ be a piecewise expanding map. Suppose that $d_n^{(r)}(x,y)$  be the $r$-th minimum of $d(x,T(y)),\cdots, d(x,T^n (y))$ for $x,\,y\in I.$ If $\rm\overline{(PE1)}$--$\rm\overline{(PE3)}$ hold,
\bea\label{PemRec}\nu\left\{x\in I:\,\,\limsup_{n\rightarrow \infty} \frac{\left|\ln d_n^{(r)}(x,x)\right|-\ln n}{\ln \ln n}
        = \frac{1}{r}\right\}=1.\eea
  In particular, \eqref{PemRec} holds for tent map, Lorentz-like map and Gauss map.
\end{Thm}

\begin{Rem}
\begin{itemize}
\item[(i)]  Tent map, Lorentz-like map and Gauss map satisfy {\rm (PE1)--(PE4)} and  $\rm\overline{(PE1)}$--$\rm\overline{(PE3)}$ with the measure $\mu=\nu.$
\item[(ii)]  Condition~\eqref{bddisbdy} implies the bounded distortion property \eqref{pebdt}.
\item[(iii)] By assumption $\rm\overline{(PE2)},$ there exists $\tau>0$ such that
\bea\label{probdy}\nu\left\{x\in I\setminus\mathcal{S}:|T'(x)|>t\right\}<t^{-\tau}\eea
for $t$ sufficiently large.
\item[(iv)] We can obtain $\rm{(PE3)}$ directly from $\rm\overline{(PE2)}.$
\end{itemize}
\end{Rem}

\subsection{Proof of Theorem~\ref{PemMultiLogHit}}\label{PemMula}
If $\psi\in L^1(I),$ the variation  $\text{Var}(\psi)$ of $\psi$ is defined by
$$\text{Var}(\psi)=\sup\sum_{i=1}^n\left|\psi(x_{i-1})-\psi(x_i)\right|,$$
where the supremum is taken over all finite partitions $0=x_0<x_1<\cdots<x_n=1.$ One say that $\psi$ has bounded variation if $\text{Var}(\psi)<\infty.$
Let ${\rm BV}({I})$ the space of functions having bounded variation on ${I}$ endowed with the norm
$$\|\psi\|_{\mathbb{BV}}:=\|\psi\|_{L^1}+\text{Var}(\psi).$$
Under the assumptions  (PE1)--(PE4), there are \cite[Corollary 3.6]{Viana97} constants $C>0$ and $\theta<1$ such that
\begin{equation}
\label{PWM2}
\left| \mathbb{E}\left(\psi_1(\psi_2 \circ T^n)\right) - (E\psi_1)( E\psi_2)\right| \leq  C
\|\psi_1\|_{\mathbb{BV}} \|\psi_2\|_{L^1} \theta^n, \,n\in\mathbb{Z}_+
\end{equation}
for any $\psi_1,\,\psi_2\in \mathbb{BV}({I}).$
Then using  \eqref{PWM2}, we obtain the following multiple exponential mixing properties \cite[Lemma 2.4]{DL23}
\bea\label{PWMr}
\left(\prod_{j=1}^{q-1} \left(\|\psi_j\|_{L^1}-C\|\psi_j\|_{\mathbb{BV}}\theta^{k_{j+1}-k_j} \right)\right)\|\psi_q\|_{L^1}\leq \mathbb{E}\left(\prod_{j=1}^q \psi_j  \circ T^{k_j }\right)\eea
$$\leq
\left(\prod_{j=1}^{q-1} \left(\|\psi_j\|_{L^1}+C\|\psi_j\|_{\mathbb{BV}}\theta^{k_{j+1}-k_j} \right)\right)\|\psi_q\|_{L^1}$$
for $\psi_j\in \mathbb{BV}({I})$ and $\psi_j\geq0,$ $1\leq j\leq q.$

Considering the potential for the derivative of piecewise expanding maps to approach infinity, the condition $\rm{(Gr)}$ as outlined in  Definition~\ref{def.mixing} is not  satisfied in this case.  Consequently, we undertake the direct verification of conditions $(GM1)_r,$ $(GM2)_r$ and $(GM3)_r$​  using \eqref{PWMr}. We now proceed to demonstrate that Diophantine points as delineated in Definition~\ref{Diop} constitute a set of full measure.

\begin{Prop}\label{DioPem}
Under the assumptions {\rm (PE1)--(PE4)},  $\mu$-almost every point is  Diophantine.
\end{Prop}
\begin{proof} Recall that $\partial \Delta=\mathcal{S}.$
Under (PE4), for any $\chi>0,$ we have
 $$ \sum_{n=1}^{\infty}\mu\left\{ x: d\left(T^n (x),\mathcal{S}\right)< e^{-\chi n} \right\} < \sum_{n=1}^\infty C e^{-\gamma\chi n} < \infty.  $$
 By Borel-Cantelli Lemma, for a.e. $x\in I$, $\exists\, n_0 (x)$, $n> n_0(x)$,
 \bea\label{disb} d\left(T^n (x),\mathcal{S}\right) >e^{-\chi n}. \eea

 Suppose that $x$ satisfies \eqref{RELY} and \eqref{disb}. Then there exist $\chi_0\geq \chi$ and $\Lambda_1>\Lambda_0$ such that
 $$   d\left(T^n (x),\mathcal{S}\right)> e^{-\chi_0n}\quad
 \text{and}\quad|(T^n)'(x)|\leq e^{n\Lambda_1},\quad\forall n\in \mathbb{Z}_+.$$
Let $\beta= \beta(x) = C+\chi_0+\Lambda_1,$ where $C$ is the constant given by \eqref{pebdt}. We take the refinement $\Delta^n$ of $\Delta$ by $T^n:$
 $$\Delta^n = \{\Delta_{i_1, \ldots, i_n} \subseteq I:\,\,\Delta_{i_1, \ldots, i_n} = \Delta_{i_1} \cap T^{-1} \Delta_{i_2} \cap \cdots \cap T^{-(n-1)} \Delta_{i_n} \}$$
 and denote by $\Delta^n(x)$ the member of $\Delta^n$ containing $x.$

 We claim that $B(x,e^{-\beta n}) \subset \Delta^n (x),$ $\forall n\in\mathbb{Z}_+.$

First, the fact that
 $$ d(x,y) < e^{-\beta n} \leq e^{-\chi_0} < d\left(x,\mathcal{S}\right)$$ dictates that
  $y\in \Delta(x)$.

Then by bounded distortion property of $T,$ we have
$$  d\left(T(x),T(y)\right)<e^C\left|T'(x)\right|d(x,y) \leq  e^{C+\Lambda_0}e^{-\beta n} \leq e^{-2\chi_0} < d\left(T(x), \mathcal{S}\right), $$
 which implies $Ty \in \Delta(Tx)$.

By induction, we obtain
$$  d\left(T^n(x),T^n(y)\right)<e^C\left|(T^n)'(x)\right|d(x,y) \leq e^{C+\Lambda_0n}e^{-\beta n} \leq e^{-\chi_0n } < d\left(T^n (x), \mathcal{S}\right), $$
which implies that $T^n(y) \in\Delta(T^n x)$ and  $B(x,e^{-\beta n}) \subset \Delta^n (x).$

On the other hand, since (see, e.g. \cite{HA02} Lemma 3.1 in Chapter 8)
\bea\label{PeEnt}h_\mu(T)=h_\mu(T,\Delta)=\int\log\left|T'(x)\right|\rd\mu(x)
=\Lambda_0\geq\log\lambda>0\eea
and by Proposition \ref{prop1}, for almost all $x\in I,$
$$\liminf_{\rho\rightarrow \infty} \frac{\tau(B(x,\rho)) }{-\log \rho}  = \liminf_{n\rightarrow \infty} \frac{\tau(B(x,e^{-\beta n})) }{\beta n} \geq \liminf_{n\rightarrow \infty} \frac{\tau(\Delta^n(x))}{\beta n}  \geq \frac{1}{\beta}.$$

Therefore, $B(x,\rho) \cap T^{-k} B(x,\rho) = \emptyset$ whenever $k< -\frac{1}{2 \beta} \log \rho$ and $\rho$ sufficiently small, which proves $\mu$-almost every point $x\in I$ is Diophantine.
\end{proof}

We take $$E_\rho=B(x, \rho)=\left\{y:d(x,y)<\rho\right\}
\,\,\,\text{for}\,\,\,x\in {I}\setminus\mathcal{S}\,\,\,\text{and}\,\,\,\rho<d\left(x,\mathcal{S}\right),$$
$$E_\rho^k=T^{-k}E_\rho\quad\text{and}\quad\rho_n=n^{-1}(\ln n)^{-\delta}\quad\text{for}\quad\delta>0.$$  Observe that $\|1_{E_\rho}\|_{\mathbb{BV}}=2$ and $\sigma(\rho)=\mu(E_\rho)=O(\rho).$ To verify conditions   ${(GM1)}_r-{(GM3)}_r$   for a Diophantine point $x,$ we refer to equation~\eqref{PWMr}. For the separation indices, we take $s(n)=R\ln n$  for  sufficiently large $R$ and $\hat{s}(n)=n/(4r).$

First, we prove ${(GM1)}_r.$  If $0<k_1<k_2<\dots<k_r\leq n$ are such that $k_{i+1}-k_i\geq R\ln n,$ $0\leq i\leq r-1$ with $k_0=0,$
$$\mu\left(\bigcap_{j=1}^r E^{k_j}_{\rho_n} \right)=
\mathbb{E}\left(\prod_{j=1}^r 1_{E_{\rho_n}}\circ T^{k_j} \right)\leq
\left(\prod_{j=1}^{r-1} \left(\mu\left(E_{\rho_n}\right)+C\theta^{R\ln n}\left(2+\mu(E_{\rho_n})\right)\right)\right)\mu\left(E_{\rho_n}\right)$$
$$\leq C\sigma({\rho_n})^r,$$
which yields the right side of ${(GM1)}_r.$    The left side is proved  similarly.

Next, we verify ${(GM2)}_r.$  Let $\Sep(k_1, \dots k_r)=m<r.$ Since $x$ is a Diophantine point, there exists $\epsilon>0$ such that $B(x,\rho_n)\cap T^{-k}B(x,\rho_n)=\emptyset$ for $k\leq \epsilon\ln n$ provided that $n$ is sufficiently large. Thus
we focus on the case $\min_j\left(k_j-k_{j-1}\right)>\epsilon\ln n$ and obtain
$$
 \mu\left(\bigcap_{j=1}^r E^{k_j}_{\rho_n} \right)
\leq \left(\prod_{j=1}^{r-1} \left(\mu\left(E_{\rho_n}\right)+C\theta^{k_{j+1}-k_j}\left(2+\mu(E_{\rho_n})\right)\right)\right)\mu\left(E_{\rho_n}\right)$$
$$\leq C\sigma({\rho_n})^m(\theta^{\epsilon\ln n})^{r-m}\leq{\frac{C  \fv(\rho_n)^m}{(\ln n)^{100r}}}.
$$

Finally, we prove ${(GM3)}_r.$   If $0<k_1<k_2<\dots<k_r<l_1<l_2<\dots<l_r$ are such that  $2^i<k_\alpha\leq 2^{i+1}, 2^j<l_\beta\leq 2^{j+1}$ for $1\leq\alpha,\beta\leq r$, $j-i\geq b$ for some constant $b\geq 2,$ and such that
$k_{\alpha+1}-k_\alpha\geq 2^{i-1}/r,$ $l_{\beta+1}-l_\beta\geq 2^{j-1}/r$ and $l_1-k_r\geq 2^{j-1}/r$
with $k_0=0$ and $l_0=0,$
\Bea
&&\mu\left(\left[\bigcap_{\alpha=1}^r E^{k_\alpha}_{\rho_{2^i}}\right]\bigcap
\left[ \bigcap_{\beta=1}^r E^{l_\beta}_{\rho_{2^j}} \right]\right)\\
&\leq&\mathbb{E}\left(\left(\prod_{\alpha=1}^r 1_{E_{\rho_{2^i}}}\circ T^{k_\alpha} \right)\left(\prod_{\beta=1}^r 1_{E_{\rho_{2^j}}}\circ T^{k_\beta} \right)\right)\\
&\leq&
\left(\prod_{\alpha=1}^r \left(\mu(E_{\rho_{2^i}})+C\theta^{2^{i-1}/r}(\mu(E_{\rho_{2^i}})+2)\right)\right)
\left(\prod_{\beta=1}^{r-1} \left(\mu(E_{\rho_{2^j}})+C\theta^{2^{j-1}/r}(\mu(E_{\rho_{2^j}})+2)\right)\right) \mu\left(E_{\rho_{2^j}}\right)\\
&\leq&C\sigma({\rho_{2^i}})^r\sigma({\rho_{2^j}})^r,\\
\Eea
Hence, Theorem~\ref{PemMultiLogHit} follows from Theorem~\ref{GMultiBCN} and Proposition~\ref{EssInv}.
\begin{Rem}
In our endeavor to prove Theorem~\ref{PemMultiLogHit},  we employ a method that approximates the indicator function of targets by functions with bounded distortion and equipped with the norm $\|\cdot\|_{\mathbb{BV}}.$ Using \eqref{PWMr} that is stronger than $({\rm GEM})_r,$ we have that the large separation between events serves to dominate the growth of the norm $\|\cdot\|_{\mathbb{BV}},$ which is influenced by the smaller separation between these events.  This regulation leads to the establishment of mixing properties  $(GM2)_r$ and $(GM3)_r$  for  simple admissible targets. It is noteworthy that the variation for certain piecewise expanding maps, such as Gauss map, is not bounded, leading to the non-fulfillment of condition $\rm{(Gr)}$ within the realm of functions with bounded variation.

In the next subsection, our objective is to transpose our proof methodology, delineated in Subsection~\ref{MEMS} for multiple recurrence, onto an alternative appropriate space, specifically the space of dynamically H\"older continuous functions.
\end{Rem}

\subsection{Proof of Theorem~\ref{PemMLgRec}}\label{PemMulRe}
Chen and Zhang~\cite{ChenZhang19} established the exponential decay of correlations  for dynamically H\"older continuous functions under assumptions $\rm\overline{(PE1)}$--$\rm\overline{(PE3)}.$ Their approach is outlined as follows: the  expanding map $T$ on $I$ is lifted to a hyperbolic map $\hat{T}$ on the unit square. This allows for the application of the functional analytic techniques designed for hyperbolic systems with singularities, as developed in~\cite{DemersZhang11}.

The connected component of the expanding map $\hat{T},$ denoted as $(a_{j+1},a_j)\times I$ for $0\leq j<N,$ can be partitioned into homogeneity regions.   This partition is done in a manner similar to the approach described by~\cite{ChernovMarkarian06} of dispersing billiard, which serves to regulate the expansion in the vicinity of singularities. Subsequently, one can project the extended singular set of  $\hat{T}^n$ to derive an augmented singular set, denoted as $\mathcal{S}_n^*,$ for the $n$-th iteration of $T.$ We suppose that $\mathcal{P}^*_n$ is the partition of the unit interval $I$ divided by points in $\mathcal{S}_n^*.$

We further introduce the concept of dynamical H\"older functions. To elaborate, we commence by defining the separation time
$s(x,x')$  for any two points $x,\,x'\in I$, as the minimal integer $n\geq 0$ for which $x$ and $x'$ reside in separate connected components of $\mathcal{P}^*_n.$ Note that there exists constant $C>0$ such that $$d(x,x')\leq C\lambda^{-s(x,x')}.$$

\begin{Def}[Dynamically H\"older continuous functions]\label{dynHolPe}
A function $A:I\rightarrow\mathbb{R}$ is said to be dynamically H\"older continuous if there exist  constant $K_A>0$ and $\theta_A\in(0,1)$ such that:
\bea\label{PeDHC}\left|A(x)-A(x')\right|\leq K_A \theta_A^{s(x,x')}\eea
for $x,\,x'\in I$ and  $s(x,x')$ the separation time of $x$ and $x'.$
\end{Def}

\begin{Rem}
The space of H\"older continuous functions is contained in the space of dynamically H\"older continuous functions. Namely, suppose  $A$ to be a H\"older continuous function such that
$$|A(x)-A(y)|\leq C_A d(x,y)^{\alpha_A}\quad\forall\,x,\,y\in I$$
for $C_A>0$ and $0<\alpha_A\leq 1.$ Then $A$ is dynamically H\"older continuous with $K_A=O(C_A)$ and $\theta_A=\lambda^{-\alpha_A}.$
\end{Rem}

\begin{Rem}
We may define dynamically H\"older continuous functions on $I^k$ as follows.
A function $A:I^k\rightarrow\mathbb{R}$ is said to be dynamically H\"older continuous if there are $\theta_A\in(0,1)$ and $K_A>0$ such that for any $x=(x_1,\cdots,x_k),$ $y=(y_1,\cdots,y_k),$
\begin{equation}\label{MHolPe}
|A(x)-A(y)|\leq K_A\theta_A^{\min_{1\leq i\leq k}\{s(x_i,y_i)\}}.
\end{equation}
Similarly, if we take $A$ to be a H\"older continuous function such that
$$|A(x)-A(y)|\leq C_A d(x,y)^{\alpha_A},\quad\forall\,x,\,y\in I^k$$
for $C_A>0$ and $0<\alpha_A\leq 1.$ Then $A$ is dynamically H\"older continuous with $\theta_A=\lambda^{-\alpha_A}$ and some constant $K_A$ depending on $C_A,$ $k$ and the system.
\end{Rem}

Let $\mathbb{DH}(I^k)$ represent the space of dynamically H\"older continuous functions on $I^k$ that are essentially bounded. For $A\in\mathbb{DH}(I^k),$ we define the norm as \bea\label{NomDHC}\|A\|_{\mathbb{DH}}=\|A\|_\infty+K_A,\eea
where $\|\cdot\|_\infty$ is the essential supremum norm, and $K_A$ is the minimal constant that satisfies\eqref{MHolPe}.

\begin{Thm} \cite[Theorem 1.2]{ChenZhang19} \label{decayHolPe}
	Let $\mathcal{A}=\{A_i\}_{i=1}^l$ and $\mathcal{B}=\{B_j\}_{j=1}^m$ be two sets of dynamically H\"older continuous functions.
	Suppose that there exist constants $\theta_\mathcal{A}\in(0,1),$ $K_\mathcal{A}>0,$ $\|\mathcal{A}\|_\infty>0$ and $\theta_\mathcal{B}\in(0,1),$ $K_\mathcal{B}>0,$ $\|\mathcal{B}\|_\infty>0$ such that for all $i, j,$
	$$\theta_{A_i}\leq\theta_\mathcal{A},\quad K_{A_i}\leq K_\mathcal{A},\quad \|A_i\|_\infty\leq\|\mathcal{A}\|_\infty,$$
and  $$\theta_{B_i}\leq\theta_\mathcal{B},\quad K_{B_i}\leq K_\mathcal{B},\quad \|B_i\|_\infty\leq\|\mathcal{B}\|_\infty.$$

	We consider two products $\tilde{A} = A_0 (A_1 \circ T^{t_{-1}}) \cdots (A_l \circ T^{t_{-l}})$ and $\tilde{B} = B_0 (B_1 \circ T^{t_1}) \cdots (B_m\circ T^{t_m})$,
	 where $t_{-l} < \cdots < t_{-1} <0<t_1 < \cdots <t_m$. Then there are constants
	 $C_{\mathcal{A}, \mathcal{B}} >0,$ $\theta=\theta_{\mathcal{A}, \mathcal{B}}<1$ such that for all
	 $n>0,$ we have
	\bea\label{MEMPEM} \left| \int \tilde{A} (\tilde{B} \circ T^n) \rd\nu - \int \tilde{A} \rd \nu \int \tilde{B} \rd\nu  \right| \leq
	C_{\mathcal{A}, \mathcal{B}} \; \theta^n,\eea
where $$C_{\mathcal{A}, \mathcal{B}}=C\|\mathcal{A}\|_\infty^l\|\mathcal{B}\|_\infty^m
\left(\frac{K_\mathcal{A}\|\mathcal{B}\|_\infty}{1-\theta_\mathcal{A}}+
\frac{K_\mathcal{B}\|\mathcal{A}\|_\infty}{1-\theta_\mathcal{B}}+\|\mathcal{A}\|_\infty\|\mathcal{B}\|_\infty\right).$$
\end{Thm}

In analogy with the dispersing billiard case, we consider the space of Lipschitz functions as a subspace  of dynamically H\"older continuous functions. This subspace is equipped with the norm  $\|\cdot\|_{\mathbb{DH}}$  as defined in \eqref{NomDHC}. We denote this space on $I^k$ by $\mathbb{DH_L}(I^k)$ and establish that the system  $(I^k,\mathcal{B}(I^k),\nu,T,\mathbb{DH_L}(I^k))$ is a generalized  $(r+1)$-fold exponentially mixing systems according to Definition~\ref{def.mixing}.

On the one hand,  $\rm{(Prod)}$ follows from \eqref{NomDHC} and we can take $L=\Lambda$ for  $\rm{(Gr)}.$ On the other hand, we obtain the multiple exponentially mixing property for $T$ using \eqref{MEMPEM}. More precisely, there exist constant  $C>0$ and $\theta<1$ such that for any  $A_0, A_1,\dots, A_r\in\mathbb{DH_L}(I)$
and any $r$ tuple $k_1<k_2<\dots<k_r,$ we have
 \begin{equation}\label{MEMPE}
 \left| \int \prod_{j=0}^{r} \left(A_j \circ T^{k_j} \right) \rd \nu  -
 \prod_{j=0}^{r}\int A_j \rd\nu  \right| \leq C \theta^n \prod_{j=0}^{r} \|A_j\|_{\mathbb{DH}},
 \end{equation}
 where $\DS n=\min_j (k_j-k_{j-1})$ with $k_0=0.$

For the generalized $(r+1)$-fold exponentially mixing property, we consider a larger class of functions. Specifically, we require the existence of constants $\brC>0$ and $\brtheta<1$ such that for any Lipschitz function $B\in\mathbb{DH_L}(I^{r+1})$ on $I^{r+1},$ the following inequality holds:
 \begin{equation}\label{MEMPEF}
\left|\int B(x_0,T^{k_1}x_0,\cdots,T^{k_r}x_0) \rd\nu(x_0)-\int B(x_0,\cdots,x_r)\rd\nu(x_0)\cdots \rd\nu(x_r)\right|\leq \brC\brtheta^n\left\|B\right\|_{\mathbb{DH}},
 \end{equation}
 where $\DS n=\min_j (k_j-k_{j-1})$ with $k_0=0.$

Using the same argument as presented in Appendix A of \cite{DFL22}, we can establish the equivalence between  \eqref{MEMPE} and \eqref{MEMPEF}. Essentially, we can demonstrate that if either \eqref{MEMPE} or \eqref{MEMPEF}  holds for $C^s$ functions with norm $\|\cdot\|_{C^s}$, then it also applies to Lipschitz functions with norm $\|\cdot\|_{\mathbb{DH}},$ and vice versa. This allows us to derive \eqref{MEMPEF} from \eqref{MEMPE} by decomposing $C^s$ functions equipped with the norm $\|\cdot\|_{C^s},$ in terms of a family of bases within the Sobolev space, which completes the proof of the generalized $(r+1)$-fold  exponentially mixing property for the system $(I^k,\mathcal{B}(I^k),\nu,T,\mathbb{DH_L}(I^k)).$

We denote $\tilde{I}=I\setminus\left(\bigcup_{k\in\mathbb{N}}T^{-k}\mathcal{S}\right)$ and take $$\bfE_{\rho}=\left\{(x,y)\in\tilde{I}\times\tilde{I}:
d(x,y)\leq\min\left\{\rho,\,d(x,\partial\tilde{I})\right\}\right\},\, \bfE_\rho^k=\left\{x:d(x,T^kx)\in\bfE_{\rho}\right\},$$ $$\bar{\fv}(\rho)=(\mu\times\mu)(\bfE_\rho),\quad\text{and}\quad\rho_n=n^{-1}(\ln n)^{-\delta},\,\,\delta>0.$$

We now demonstrate that the sequence $\{\bfE_{\rho_n}\}$ constitutes composite admissible targets as Definition \ref{def.recurrent.targets}. The proof is structured as follows:
 \begin{itemize}
\item[$\bullet$] Observe that $\bsigma(\rho)=(\nu\times\nu)\left(\bfE_{\rho}\right)=O\left(\rho\right)$ and
 $$\int\prod_{i=1}^r\left(\int 1_{\bfE_\rho}(x_0,x_i) \rd\nu(x_i)\right) \rd\nu(x_0)=O\left(\rho^r\right).$$
\item[$\bullet$] We select functions $\bA_\rho^+,\,\bA_\rho^-\in\mathbb{DH_L}(I^k)$ such that $$1_{\bfE_{\rho-\rho^s}}\leq \bA_\rho^-\leq 1_{\bfE_\rho}\leq \bA_\rho^+\leq 1_{\bfE_{\rho+\rho^s}},\,\,  \|\bA_\rho^\pm\|_{\mathbb{DH}}\leq C\,\,\text{for some constants}\,\,s>1,\,\,C>0.$$
\item[$\bullet$] Then $$\mathbb{E}\bA_\rho^+-\mathbb{E}\bA_\rho^-\leq  C\bar{\fv}(\rho)^{(1+s)/2},$$ which verifies  ${\rm(\overline{Appr})(i)-(iii)}.$
\item[$\bullet$]     Note that
$$\int\bar{A}_\rho^+(x,y)\rd\nu(y)\leq\nu (B(x,{\rho+\rho^s}))\leq C\rho$$
and a similar inequality holds for $\DS\int\bar{A}_\rho^+(x,y)d\nu(x).$ Then ${\rm(\overline{Appr})(iv)-(v)}$ follows by $$\int\bar{A}_\rho^+(x,T^kx)\rd\nu(x)\leq\nu\left(\bfE_{\rho+\rho^s}^k\right)\leq\nu\left(\bfE_{2\rho}^k\right).$$
\item[$\bullet$] Next $({\rm \overline{Sub}})$ is valid because $$\bfE_{\rho}^{k_1}\cap\bfE_{\rho}^{k_2}\subset T^{-k_1}\bfE_{2\rho}^{k_2-k_1}.$$
\item[$\bullet$] Finally, the following proposition verifies $({\rm \overline{Mov}})$.
\end{itemize}

\begin{Prop}\label{SlowREc-Pe}
For each $A>0,$ there exists $\rho_0>0$ such that for all $\rho<\rho_0$ and all $k\in \mathbb{Z}_+,$ we have
\begin{equation}
\label{EqWNRS}
 \nu\left(\left\{x:d\left(x,T^k(x)\right)<\rho\right\}\right)\leq C|\ln \rho|^{-A}.
\end{equation}
\end{Prop}

\begin{proof} We prove \eqref{EqWNRS} for sufficiently large $A,$ then \eqref{EqWNRS} holds for each $A>0.$
Take $B=\max\{A^2,A/\tau\},$ where $\tau$ is defined in \eqref{probdy}.  If $k\geq B\ln|\ln\rho|,$ take $\hat\rho=|\ln \rho|^{-A}.$ By $\rm{(GEM)_1},$ we obtain that
\begin{multline}
\label{eq.highPe}
\nu \left(\left\{x: d\left(x,T^k(x)\right)\leq \rho\right\}\right)\leq\nu \left(\left\{x: d\left(x,T^k(x)\right)\leq \hat\rho\right\}\right)\\
\leq \int\bar{A}^+_{\hat{\rho}}\left(x,T^k(x)\right)\rd\nu(x) \leq C\left((\hat\rho+\hat{\rho}^s)^2+\hat\rho^{-\tau} \theta^k\right)\leq |\ln \rho|^{-2A},
\end{multline}
provided that $A\geq-2(s+1)/\ln\theta$ and $\rho$ is sufficiently small.

Now fix any $1\leq k\leq B\ln|\ln\rho|$.  Assume that $x$ satisfies $d\left(x,T^k(x)\right)\leq \rho.$ We denote
$$\mathcal{I}_n=\left\{x\in I\setminus\mathcal{S}:|T'(x)|\leq n\right\}\quad\text{and}\quad\mathcal{I}_{n,l}=\bigcap_{j=1}^l T^{-j}\mathcal{I}_n.$$
Then we have that $$\nu\left(\mathcal{I}_{n,l}^c\right)\leq\sum_{j=1}^l\nu(\mathcal{I}_n^c)\leq ln^{-\tau}$$
by \eqref{probdy}.

For any $x\in\mathcal{I}_{n,l}$ and $l\geq jk,$ we obtain that
$$d\left(T^{(j-1)k}(x), T^{jk} (x)\right)\leq n^{(j-1)k} \rho.$$
If we take $n=\lfloor|\ln\rho|^{2B}\rfloor,$ $l=\lfloor8B\ln|\ln\rho|/k\rfloor,$ $L=\lceil4B\ln|\ln\rho|/k\rceil,$ we find
$$d\left(x,T^{Lk} (x)\right)\leq \sum_{j=1}^L d\left(T^{(j-1)k}(x),T^{jk}(x)\right)
\leq L n^{4B\ln|\ln\rho|}\rho\leq \sqrt{\rho},$$
provided that $\rho$ is sufficiently small. Since $Lk\geq B\ln|\ln\rho|,$  \eqref{eq.highPe} is applied and we have
$$\nu \left(\left\{x: d\left(x, T^k (x)\right)\leq \rho\right\}\right)
\leq \nu(\mathcal{I}_{n,l}^c)+\nu \left(\left\{x\in\mathcal{I}_{n,l}: d(x, T^k (x))\leq \rho\right\}\right)$$
$$\leq l n^{-\tau}+\nu\left(\left\{x: d(x, T^{Lk} (x))\leq \sqrt\rho\right\}\right)\leq C|\ln \rho|^{-A}.
$$
\end{proof}

We have obtained that $\left({I},\mathcal{B}({I}),T,\mu,\mathbb{BV}({I}^{r+1})\right)$ is generalized $(r+1)$-fold exponentially mixing
as Definition~\ref{def.mixing} and $\{\bfE_{\rho_n}\}$ is composite admissible as Definition~\ref{def.recurrent.targets}. By Corollary~\ref{cor.mixing},
$$\limsup_{n\rightarrow \infty} \frac{\left|\ln d_n^{(r)}(x,x)\right|-\ln n}{\ln \ln n}\leq \frac{1}{r}\,\,\text{for}\,\, a.e.\text{-}x$$
and
$$\limsup_{n\rightarrow \infty} \frac{\left|\ln d_n^{(r)}(x,x)\right|-\ln n}{\ln \ln n}\geq \frac{1}{r}\,\,\text{on a positive measure set}.$$

If we denote
$$R_{\rm{Pe}}=\left\{x\in I:\limsup_{n\rightarrow \infty} \frac{\left|\ln d_n^{(r)}(x,x)\right|-\ln n}{\ln \ln n}
        = \frac{1}{r}\right\},$$
then we have $\mu\left(R_{\rm{Pe}}\right)>0.$ We now prove  $\mu\left(R_{\rm{Pe}}\right)=1.$

Suppose that $x\in R_{\rm{Pe}}\setminus\left(\bigcup_{k\in\mathbb{N}}T^{-k}\mathcal{S}\right).$ For any $\epsilon>0,$ there exists an increasing sequence $\{n_k\}_{k=1}^\infty$ such that
$$d_{n_k}^{(r)}(x,x)<\frac{1}{n_k\left(\ln\ln n_k\right)^{\frac{1}{r}-\epsilon}},\quad\forall k\in\mathbb{Z}_+$$
and $\left|T'(y)\right|\leq e^C\left|T'(x)\right|$ for $x$ and $y$ in the same continuity domain of $T.$

It follows that
$$d_{n_k}^{(r)}\left(T(x),T(x)\right)<\frac{e^C\left|T'(x)\right|}{n_k\left(\ln\ln n_k\right)^{\frac{1}{r}-\epsilon}},\quad\forall k\in\mathbb{Z}_+$$
and
$$\limsup_{n\rightarrow \infty} \frac{\left|\ln d_n^{(r)}\left(T(x),T(x)\right)\right|-\ln n}{\ln \ln n}\geq \frac{1}{r},\quad\forall k\in\mathbb{Z}_+.$$
Hence $R_{\rm{Pe}}\setminus\left(\bigcup_{k\in\mathbb{N}}T^{-k}\mathcal{S}\right)\subseteq T^{-1}R_{\rm{Pe}}.$ In addition since $\nu\left(\mathcal{S}\right)=0,$ we have $\nu\left(R_{\rm{Pe}}\Delta T^{-1}R_{\rm{Pe}}\right)=0.$
Then Theorem~\ref{PemMLgRec} follows by the ergodicity of the measure $\nu.$

\appendix
\section{Proof of Theorem~\ref{GMultiBCN}}\label{appmbc}
Let  $a_2>a_1>0$ and $a_2\in\mathbb{Z}_+.$ We estimate the probability of the events that there exist $r$ integers $k_1<k_2<\cdots<k_r$​ such that $a_1n\leq k_j\leq a_2n$ and  $\{E^{k_j}_{\rho_n}\}_{j=1}^r$ occur. To simplify the notation, we define
$$ E^{k_1, \dots, k_r}_{\rho_n}:=\bigcap_{j=1}^r  E^{k_j}_{\rho_n}$$
for a family of events $\{ E_{\rho_{n}}^k \}_{(n,k) \in \mathbb{N}^2; 1\leq k \leq 2n}.$

\begin{Lem}
\label{LmSum}
If  $(GM1)_r$ and $(GM2)_r$ are satisfied, then there exist constant $C>0$ and sequence $\eta_n\rightarrow0$ such that
  \begin{equation}
\label{EqSum}
C^{-1}n^r \fv(\rho_n)^r-\eta_n\left(\ln n\right)^{-10}\leq\sum_{a_1 n<k_1<k_2<\dots <k_r\leq a_2n}
\Prob\left(E^{k_1, \dots, k_r}_{\rho_n}\right)\leq C n^r \fv(\rho_n)^r+\eta_n\left(\ln n\right)^{-10}.
\end{equation}
In addition if $a_2-a_1\geq1/2,$ then
\begin{equation}
\label{WellSepHat}
C^{-1} n^r \fv(\rho_n)^r\leq\sum_{\overset{a_1 n<k_1<k_2<\dots <k_r\leq a_2 n}{ { \wSep_n}(k_1,\dots, k_r)=r}} \Prob\left(E^{k_1, \dots, k_r}_{\rho_n}\right)\leq C n^r \fv(\rho_n)^r.
\end{equation}

\end{Lem}

\begin{proof}
Suppose that $\{\delta_n\}$ is a decreasing sequence such that  $\rho_n=\delta_{a_2n}.$ If the targets $\{E_{\delta_n}\}_{n=1}^\infty$ satisfy conditions $(GM1)_r$ and $(GM2)_r$ with a separation function $s_\delta(n)\leq C(\ln n)^2,$ then  $(GM1)_r$ and $(GM2)_r$ also hold for the targets  $\{E_{\rho_n}\}_{n=1}^\infty$ with the conditions replaced by $0<k_1<k_2<\cdots<k_r\leq a_2n$ for the separation function  $s(n)= s_\delta(a_2n),$ which satisfies $s(n)\leq Ca_2^2(\ln n)^2.$

We denote by
\begin{equation*}
S_m:=\sum_{\overset{a_1 n<k_1<k_2<\dots <k_r\leq a_2 n}
{ { \Sep_n}(k_1,\dots, k_r)=m}} \Prob\left(E^{k_1, \dots, k_r}_{\rho_n}\right),\,\,\,m\leq r.\end{equation*}
 Note that $S_r$ includes
$(a_2-a_1)^rn^r(1+\delta'_n)/r!$ terms for some sequence $\delta'_n\rightarrow0$ as $n\to\infty,$  hence $(GM1)_r$ yields
\begin{equation}
\label{WellSep}
C^{-1}n^r \fv(\rho_n)^r\leq S_r \leq Cn^r \fv(\rho_n)^r\end{equation}
for some constant $C>0.$

For $m<r,$
$S_m$ consists of $O\left(n^m s(n)^{r-m}\right)$ terms.
Hence $(GM2)_r$  gives that
\begin{equation}
\label{NonWellSep}
 S_m\leq Cn^m s(n)^{r-m} \frac{\fv(\rho_n)^m}{(\ln n)^{100r}} =\eta_n n^m \fv(\rho_n)^m \left(\ln n\right)^{-10}
\end{equation}
for some sequence $\eta_n\rightarrow0.$
Combining \eqref{WellSep} with \eqref{NonWellSep} we obtain \eqref{EqSum}.

The proof of \eqref{WellSepHat}  parallels that of \eqref{WellSep} except for one difference. In \eqref{WellSep}, the  sum includes $O\left((a_2-a_1)^rn^r(1+\delta'_n)/r!\right)$ terms. In \eqref{WellSepHat}, the number of terms in the sum is bounded by $(a_2-a_1)^r\left(n/2-r\hat s(n)\right)^r/r!,$ which is  greater than $(a_2-a_1)^rq^rn^r/(2^rr!)$,  by the assumption that $\hat{s}(n) <n(1-q)/{(2r)}$.
\end{proof}

For $m \in \N,$ let
$$\mathcal U_m=\left\{(k_1,\ldots,k_r):\,\,\,2^m<k_1<k_2<\dots <k_r\leq 2^{m+1} \textrm{ and } \wSep_{2^{m+1}}(k_1,\dots k_r)=r\right\}$$ and
\begin{align}
\label{dAm} \cA_m&=\bigcup_{0<k_1<k_2<\dots <k_r\leq 2^{m+1}} E^{k_1, \dots, k_r}_{\rho_{2^m}},\\
\label{dDm}
 \cD_m&=\bigcup_{(k_1,\ldots,k_r) \in \cU_m}  E^{k_1, \dots, k_r}_{\rho_{2^{m+1}}}.
\end{align}

\begin{Prop}\label{prop6}
Suppose that
\begin{equation}
\label{SmallTargets}
n\fv(\rho_n)\to 0\quad\mathrm{as}\quad n\to\infty.
\end{equation}
Then there exists constant $C>0$  such that:

(i) If $(GM1)_r$ and $(GM2)_r$ are satisfied, then
\begin{equation}
\label{Am}
\Prob(\cA_m)\leq C\left( 2^{rm} \fv(\rho_{2^{m+1}})^r+{m^{-10}} \right);
\end{equation}

(ii)
If $(GM1)_k$ and $(GM2)_k$ are satisfied for $k=r,r+1,\dots, 2r,$ then
\begin{equation}
\label{Dm}
 \Prob(\cD_{m})\geq C^{-1} \left(2^{rm} \fv(\rho_{2^{m+1}})^r -m^{-10}\right);
\end{equation}

(iii) If $(GM1)_k$ and $(GM2)_k$ are satisfied for $k=r,r+1,\dots, 2r,$ $(GM3)_r$ is satisfied and  $m'-m\geq b,$ where $b$ is a constant from $(GM3)_r,$ then
\begin{equation}
\label{AsymInd}
 \Prob\left(\cD_m\cap \cD_{m'}\right)\leq  C\left(\Prob(\cD_m)+m^{-10}\right) \left(\Prob(\cD_{m'}) +m'^{-10}\right).
\end{equation}
\end{Prop}

\begin{proof}[Proof of Proposition \ref{prop6}] First, we get \eqref{Am} using \eqref{dAm} and  \eqref{EqSum}.  Next, we
 denote
\begin{align*} I_m&=\sum_{(k_1,\dots k_r)\in \cU_m}\Prob\left(E^{k_1, \dots, k_r}_{\rho_{2^{m+1}}}\right),\\
J_m&=\sum_{\overset{(k_1,\ldots,k_r)\in \cU_m}{\overset{(k'_1,\ldots,k'_r)\in \cU_m}{\{k_1,\dots, k_r\}\neq \{k'_1,\dots, k'_r\}} }}
\Prob\left(E^{k_1,\dots, k_r}_{\rho_{2^{m+1}}}\bigcap E^{k'_1,\dots, k'_r}_{\rho_{2^{m+1}}}\right). \end{align*}
From \eqref{dDm} and Bonferroni inequalities we obtain that
\begin{equation} \label{eIJ}
I_m-J_m\leq \Prob\left(\cD_m\right)\leq I_m
\end{equation}

Besides, \eqref{WellSepHat} implies that
\begin{equation} \label{eI}
C^{-1} 2^{r(m+1)} \fv(\rho_{2^{m+1}})^r\leq I_m\leq C 2^{r(m+1)} \fv(\rho_{2^{m+1}})^r.
\end{equation}
On the other hand, since
$$ E^{k_1,\dots, k_r}_{\rho_{2^{m+1}}}\bigcap E^{k'_1,\dots, k'_r}_{\rho_{2^{m+1}}}=E^{\{k_1,\dots, k_r\}\cup\{k'_1,\dots, k'_r\}}_{\rho_{2^{m+1}}}, $$
we have that
$$J_m\leq C  \sum_{l=r+1}^{2r} \sum_{2^m<k_1<\dots< k_l\leq 2^{m+1}}
\Prob\left(E^{k_1,\dots, k_l}_{\rho_{2^{m+1}}}\right),$$
and \eqref{EqSum} then implies that
\begin{equation}\label{eJ} J_m\leq C 2^{(r+1)m} \fv (\rho_{2^{m+1}})^{r+1}+\eta_m m^{-10}.
\end{equation}
Thus, by \eqref{eIJ}, \eqref{eI} and \eqref{eJ},
we obtain \eqref{Dm}.

Finally, we observe that
$$
\Prob\left(\cD_m\cap \cD_{m'}\right)
 \leq    \sum_{(k_1,\ldots,k_r) \in \cU_m, (l_1,\ldots,l_r) \in \cU_{m'}}\Prob\left(E^{k_1,\ldots,k_r}_{\rho_{2^{m+1}}} \cap E^{l_1,\ldots,l_r}_{\rho_{2^{m'+1}}}\right).
$$
But since $m'>m+1$ implies that $l_1-k_r \geq \hat s (2^{m'+1})$,  $(GM3)_r$ then yields
\begin{equation*}\label{ApplyM3}
\Prob\left(E^{k_1,\ldots,k_r}_{\rho_{2^{m+1}}} \cap E^{l_1,\ldots,l_r}_{\rho_{2^{m'+1}}}\right)\leq C\Prob\left(E^{k_1,\ldots,k_r}_{\rho_{2^{m+1}}}\right) \Prob\left(E^{l_1,\ldots,l_r}_{\rho_{2^{m'+1}}}\right),
\end{equation*}
so that using $(GM1)_r$ and summing over all $(k_1,\ldots,k_r) \in \cU_m, (l_1,\ldots,l_r) \in \cU_{m'}$ we have that
\begin{equation}\label{DmIm}\Prob\left(\cD_m\cap \cD_{m'}\right)
 \leq C I_m I_{m'}. \end{equation}
In addition, by \eqref{Dm} and \eqref{eI},
\begin{equation}\label{ImBd}I_m\leq C\left(\Prob(\cD_m)+Cm^{-10}\right).\end{equation}
Then \eqref{AsymInd} follows from \eqref{DmIm} and \eqref{ImBd}.
\end{proof}

\subsection{ Convergent case. Proof of Theorem \ref{GMultiBCN} (a).}
Suppose that $\bS_r<\infty.$  Then due to the monotonicity of $\sigma(\rho_n),$ we have $n \sigma(\rho_n)\to 0$ as $n\to\infty.$
Using \eqref{Am} from Proposition~\ref{prop6}, we obtain $ \sum_m \Prob(\cA_m) < \infty$.
Thus by Borel-Cantelli Lemma, we conclude that with probability one, the events $\{\cA_m\}_{m=1}^\infty$ occur at most finitely many times. Note that for $n \in (2^m,2^{m+1}]$, $\left\{\omega: N^n_{\rho_n}(\omega) \geq r\right\} \subset \cA_m $
since $E_{\rho_n}^k\subset E_{\rho_{2^{m}}}^k.$
Therefore, on a full measure set, the events $\left\{\omega: N^n_{\rho_n}(\omega)\geq r\right\}$
also occur at most finitely many times. This leads to the conclusion that $\Prob(H_r)=0.$

\subsection{ Divergent case. Proof of Theorem \ref{GMultiBCN} (b).}
If $\bS_r=\infty,$ we provide a proof under the assumption \eqref{SmallTargets}.
 The case where \eqref{SmallTargets}
 does not hold requires only minimal modifications,  which will be delineated at the end of this subsection.

Let $\DS Z_n = \sum_{m=1}^{n} 1_{\cD_m}$ and we compute the expectation of $Z_n^2.$
\begin{equation}
\label{VarCount}
 \mathbb{E}Z_n^2=\sum_{m=1}^n \Prob\left(\cD_m\right)
 +2\sum_{i<j} \Prob\left(\cD_i \cap \cD_j\right)=\mathbb{E}Z_n+2\sum_{i<j} \Prob\left(\cD_i \cap \cD_j\right).
\end{equation}
By \eqref{AsymInd}
if $i\geq b$ and $j-i\geq b$ then
\begin{equation}
\label{DeltaBound}
\Prob\left(\cD_i \cap \cD_j\right)\leq C\left(\Prob(\cD_i)+i^{-10}\right)\left(\Prob(\cD_j)+j^{-10}\right).
\end{equation}
It follows that
\begin{equation}
\label{DeltaBound}
 \sum_{{i\geq b,\,\,}
{ { j-i\geq b}} }\Prob\left(\cD_i \cap \cD_j\right)\leq C\left[(\mathbb{E}Z_n)^2+8\mathbb{E}Z_n+8\right].
\end{equation}
Moreover, the terms where $i\leq b$ or $j-i\leq b$ contribute at most
$$[2m(\delta)+1] \sum_{j=1}^n \Prob(\cD_j)=\left[2m(\delta)+1\right]\EXP Z_n. $$
Since $\EXP Z_n=\sum_{m=1}^n \Prob(\cD_m) \geq C^{-1}\sum_{m=1}^n 2^{r(m+1)} \fv(\rho_{2^{m+1}})^r\to \infty$ as $n\rightarrow\infty,$ then
$$ \mathbb{E}Z_n^2\leq C(\mathbb{E}Z_n)^2$$
for $C$ independent of $n.$

Let us define $\tilde{Z}_n=Z_n/\mathbb{E}Z_n$ and select a weak limit $\tilde{Z}$ in the ball of radius $C^{1/2}$ of $L^2$ functions. Since $\tilde{Z}_n\rightarrow_{L^1}\tilde{Z}$  after passing to a subsequence, we have  $\mathbb{E}\tilde{Z}=1.$ It follows that $\tilde{Z}$ must be positive on a set of positive measure. Furthermore, combining the fact $\mathbb{E}Z_n\rightarrow\infty$ with $\cD_m\subset\left\{\omega:N^{2^{m+1}}_{\rho_{2^{m+1}}}(\omega) \geq r \right\},$ we deduce that $\Prob(H_r)>0.$

Note that if $\bS_r<\infty,$ then  \eqref{SmallTargets} is satisfied, which validates the proof of Theorem \ref{GMultiBCN}(a). Therefore, we will focus on the case where $\bS_r=\infty$ and \eqref{SmallTargets} fails. After passing to a subsequence, we  choose a decreasing sequence $\nu_n$ such that $\tilde{\sigma}(\rho_n):=\nu_n \sigma(\rho_n)$ satisfies
$\DS \lim_{n\to \infty} n\tilde{\sigma}(\rho_n)=0$ and  $\sum_j 2^{rj}\tilde{\sigma}(\rho_{2^j})^r=\infty.$

For each positive integer $n$​ and  $k \leq 2n,$  we define a sequence of events $\{\tilde E^{k}_{\rho_{n}}\}_{1\leq k\leq 2n}$ such that​ the following conditions hold:
\begin{itemize}
\item[(1)] If $E^{k}_{\rho_n}$​ does not occur, then $\tilde E^{k}_{\rho_{n}}$​ does not occur;
\item[(2)] If $E^{k}_{\rho_n}$​ occurs, $\tilde E^{k}_{\rho_{n}}$ occurs with a probability $\nu_n,$
independently of all other events (i.e., all other $E^{k}_{\rho_n}$​ with different $k$ or different  $n$).
\end{itemize}

 Then the events $\{\tilde{E}_{\rho_{n}}^k\}_{1\leq k \leq 2n}$
 satisfy  $(GM1)_r$, $(GM2)_r$, and $(GM3)_r$ the same way as
 the events $\{E_{\rho_{n}}^k\}_{1\leq k \leq 2n}$, with the difference that  ${\sigma}(\rho_n)$ is now replaced with $\tilde{\sigma}(\rho_n)$.

Since condition \eqref{SmallTargets} holds for $\tilde{\sigma}(\rho_n)$ and
$\sum_j 2^{rj}\tilde{\sigma}(\rho_{2^j})^r=\infty,$ we can conclude that at least $r$ events among the events $\{\tilde{E}_{\rho_{n}}^k\}_{1\leq k\leq 2n}$ occur for infinitely many $n$ on a positive set. Consequently, on a positive set, at least $r$ events among the events $\{E_{\rho_{n}}^k\}_{1\leq k\leq 2n}$ also occur for infinitely many $n$.
 Thus, the proof of Theorem~\ref{GMultiBCN}(b) is complete. \hfill $\square$

\section{Proof of Theorem~\ref{theo.mixing}}\label{prthmmix}
Theorem \ref{theo.mixing} is a direct consequence of the following Proposition. We accept a convention that $({\rm GEM})_{k}$
for $k\leq 0$ is  always satisfied.

\begin{Prop}\label{ThEM-BC}
Suppose that  $(X,\mathcal{B},\mu,f,\mathbb{B})$  is a dynamical system and $\{E_{\rho_n}\}$ is a collection of sets in $X$
such that $({\rm Prod}),$ $({\rm Poly})$ and $({\rm Appr})$  hold, then with the function  $\fv(\cdot):=\mu(E_\cdot)$, and

(i)  If  $({\rm GEM})_{r-1}$ holds, then $(GM1)_r$ is satisfied with  the function $s: \N \righttoleftarrow:  s(n)=R \ln n$, where $R$ is sufficiently large (depending on $r$, the system and the targets).

(ii) If $({\rm Gr})$, $({\rm Mov})$ and   $({\rm GEM})_{r-2}$  hold, then $(GM2)_r$ is satisfied.

(iii) If $({\rm Gr})$ and $({\rm GEM})_r$ hold, then for arbitrary $\eps>0,$  ${(GM3)}_r$ is satisfied
with $\hat{s}(n)=\eps n.$

Similarly, suppose that  $(X,\mathcal{B},\mu,f,\mathbb{B})$  is a dynamical system and $\{\bfE_{\rho_n}\}$ is a collection of sets in $X$such that $(\rm Prod),$ $(\overline{\rm Poly})$ and $(\overline{\rm Appr})$ hold, then  with the function  $\bar \fv(\cdot):=\left(\mu\times\mu\right)(\bfE_\cdot)$:

(i) If $({\rm GEM})_r$ holds, then $(GM1)_r$ is satisfied with the function $s: \N \righttoleftarrow:  s(n)=R \ln n$, with $R$ sufficiently large (depending on $r$, the system and the targets).

(ii) If $({\rm Gr})$, $(\overline{\rm Mov}),$ $(\overline{\rm Sub})$  and $({\rm GEM})_{r-1}$  hold, then $(GM2)_r$ is satisfied.

(iii) If $({\rm Gr})$ and $({\rm GEM})_r$ hold, then for arbitrary $\eps>0,$
${(GM3)}_r$ is satisfied
with $\hat{s}(n)=\eps n.$

%Besides, the same conclusions hold for the sequence of  decreasing sets $\bfE_{\rho_n}$ and $\bfE_{\rho_n}^k$ satisfying for $k_1\leq k_2,$ $\bfE_{\rho_n}^{k_1}\cap\bfE_{\rho_n}^{k_2}\subset f^{-k_1}\bfE_{2\rho_{n}}^{k_2-k_1}$
%and the conditions \eqref{rhobound}, $({\rm Appr})$ and $({\rm Mov})$ replaced by  \eqref{brhobound}
%${\rm(\overline{Appr})}$ and ${\rm(\overline{Mov})}$ respectively.

\end{Prop}

\begin{proof}[Proof of Proposition \ref{ThEM-BC}]
(i) For $E_{\rho_n},$ we prove $(GM1)_r$ for the case $k_{i+1}-k_i\geq \sqrt{R}\ln n,$ where $R$ is a sufficiently
large constant. Indeed, using ${\rm({Appr})}$ and  ${\rm {(GEM)}_{r-1}}$ we have
$$ \mathbb{E}\left(\prod_{i=1}^r 1_{E_{\rho_n}}(f^{k_i} x)\right)\leq
\mathbb{E}\left(\prod_{i=1}^r A_{\rho_n}^+ (f^{k_i} x)\right)\leq
C\left(\left(\mathbb{E} A_{\rho_n}^+\right)^r+ \rho_n^{-r\tau} \theta^{\sqrt{R}\ln n}\right)$$
$$ \leq
C\left(\sigma(\rho_n)^r+ \rho_n^{-r\tau}   \theta^{\sqrt{R}\ln n}\right),
$$
which yields the right-hand side of $(GM1)_r$, due to ${\rm(Poly)}$ if  $R$ is sufficiently large.  The left-hand side is proved similarly.

For $\bfE_{\rho_n},$ we approximate $1_{\bar{E}_{\rho_n}}$ by $\bar A_{\rho_n}^+,$ apply ${\rm(\overline{Appr})}$ and ${\rm {(GEM)}_r}$ to the function
$$B_{\rho_n}^+(x_0,\cdots,x_r)=\bar{A}_{\rho_n}^+(x_0,x_1)\cdots \bar{A}_{\rho_n}^+(x_0,x_r)$$
and obtain
\begin{align*}
&\mathbb{E}\left(\bigcap_{j=1}^r\bar{E}^{k_j}_{\rho_n}\right)\leq C\left(\bsigma(\rho_n)^r
+{\rho_n}^{-r\tau}\theta^{\sqrt{R}\ln n}\right),\end{align*}
which yields the right-hand side of $(GM1)_r$ due to ${\rm (\overline{Poly})}$  if $R$ is taken sufficiently large. The left-hand side is proved similarly.

\medskip

(ii) For $E_{\rho_n},$ it is enough to consider the case $\Sep(k_1,\dots, k_r)=r-1$
otherwise we can estimate all $1_{E_{\rho_n}} \circ f^{k_i} $ with
$k_i-k_{i-1}<s(n)$, except the first, by 1.

Hence we assume that $0<k_j-k_{j-1}<R \ln n$ and $k_i-k_{i-1}\geq R \ln n$ for $i\neq j.$
Since $(GM1)_r$ has been proven under the assumption that $\min_i (k_i-k_{i-1})>\sqrt{R} \ln n,$ we may assume that $k_j-k_{j-1}<\sqrt{R} \ln n$. Note that by (Appr),
$$ \mathbb{E}\left(A_{\rho_n}^+  \left(A_{\rho_n}^+\circ f^k\right)\right)-
\mathbb{E}\left(1_{E_{\rho_n}} \left(1_{E_{\rho_n}}\circ f^k\right)\right) \leq 4 \mathbb{E}\left(A_{\rho_n}^+-1_{E_{\rho_n}}\right)
\leq 4 C \mu(E_{\rho_n})^{1+\eta} . $$
Therefore (Mov) implies :
$$\mathbb{E}\left(A_{\rho_n}^+  \left(A_{\rho_n}^+\circ f^{k_j-k_{j-1}}\right)\right)\leq C \mu(E_{\rho_n})(\ln n)^{-1000r}.$$
We take $B=A_{\rho_n}^+  \left(A_{\rho_n}^+\circ f^{k_j-k_{j-1}}\right)$ and obtain that
$$ \mathbb{E}\left(\prod_{i=1}^r 1_{E_{\rho_n}}\circ f^{k_i} \right)\leq
\mathbb{E}\left(\prod_{i=1}^r A_{\rho_n}^+\circ f^{k_i}\right)=
\mathbb{E}\left(\left(\prod_{i\neq j-1, j} A_{\rho_n}^+\circ f^{k_i}\right) \left(B\circ f^{k_{j-1}}\right) \right)
$$
$$\leq C\left(\left(\mathbb{E}A_{\rho_n}^+\right)^{r-2}\EXP B+{\rho_n}^{-r\tau} L^{\sqrt{R} \ln n} \theta^{R \ln n}\right) \leq C \sigma(\rho_n)^{r-1}(\ln n)^{-1000r}
$$
by ${\rm {(GEM)}_{r-2}}$ and ${\rm (Poly)},$ which proves $(GM2)_r.$

For $\bfE_{\rho_n},$ we approximate $1_{\bar{E}_{\rho_n}}$ by $\bar A_{\rho_n}^+.$ We denote
$$ \bar{B}_r(x_0,\cdots,x_{j-1},x_{j+1},\cdots,x_r)$$
$$=
1_{\bar{E}_{\rho_n}}(x_0,x_1)\cdots 1_{\bar{E}_{\rho_n}}(x_0,x_{j-1})1_{\bar{E}_{a\rho_n}^{k_j-k_{j-1}}}(x_{j-1})
1_{\bar{E}_{\rho_n}} (x_0,x_{j+1})\cdots 1_{\bar{E}_{\rho_n}}(x_0,x_r)$$
and
$$ \hB_r(x_0,\cdots,x_{j-1},x_{j+1},\cdots,x_r)$$
$$=
\bA_{\rho_n}^+(x_0,x_1)\cdots \bA_{\rho_n}^+(x_0,x_{j-1})\bA_{a\rho_n}^+ (x_{j-1},f^{k_{j}-k_{j-1}}x_{j-1})
\bA_{\rho_n}^+(x_0,x_{j+1})\cdots \bA_{\rho_n}^+(x_0,x_r). $$
Since ${\rm(\overline{Appr})},$
${\rm(\overline{Mov})}$ and ${\rm(\overline{Sub})}$ are satisfied, we obtain from ${\rm {(GEM)}_{r-1}}$ that
 \Bea
\mathbb{E}\left(\bigcap_{i=1}^r\bar{E}^{k_i}_{\rho_n}\right)
&=&
\int\prod_{i=1}^r 1_{\bar{E}_{\rho_n}}(x,f^{k_i}x)\rd\mu(x)\\
&\leq&
\int\bar{B}_r(x,\cdots,f^{k_{j-1}}x,f^{k_{j+1}}x,\cdots,f^{k_{r}}x)\rd\mu(x)\\
&\leq&
\int\hB_r(x,\cdots,f^{k_{j-1}}x,f^{k_{j+1}}x,\cdots,f^{k_{r}}x)\rd\mu(x)\\
&\leq& C\left(\mathbb{E}\hB_r + {\rho_n}^{-r\tau} L^{\sqrt{R} \ln n} \theta^{R \ln n}\right).
 \Eea
Integrating with respect to all variables except for $x_0$ and $x_{j-1}$, then integrating along $x_0$ for any fixed value of $x_{j-1}$, and finally integrating along $x_{j-1}$, we obtain from ${\rm(\overline{Appr})(iii)-(iv)}$ that
$$
  \mathbb{E}\hB_r\leq
C\bar{\sigma}(\rho_n)^{r-1}\int\bar{A}_{a\rho_n}^+(x,f^{k_j-k_{j-1}}x)\rd\mu(x),$$
then ${\rm(\overline{Appr})}$(v) yields:
 \begin{equation*} \label{IntHB}
\mathbb{E}\hB_r\leq
C\bar{\sigma}(\rho_n)^{r-1}
 \mu(\bar{E}_{a_0a\rho_{n}}^{k_j-k_{j-1}}).\end{equation*}

Hence, $(GM2)_r$ follows from  ${\rm\overline{(Mov)}}$, provided  that  $R$ is sufficiently large.

\medskip
(iii) Fix a large constant $b$ that will be given below.
Consider first the simple admissible targets $\left\{E_{\rho_n}\right\}.$
Denote by  $\DS B=\prod_{\alpha=1}^r A_{\rho_{2^i}}^+\circ f^{k_\alpha}$
for $2^i<k_1<\cdots<k_r\leq 2^{i+1},$
we obtain from (Prod), (Gr), (Appr), (Poly) and ${\rm {(GEM)}_{r}}$
that $\DS \|B\|_\mathbb{B}\leq C L^{r 2^{i+1}}.$ Thus
\Bea
&&\mathbb{E}\left(\left(\prod_{\alpha=1}^r 1_{E_{\rho_{2^i}}}\circ f^{k_\alpha}\right)
\left(\prod_{\beta=1}^r 1_{E_{\rho_{2^j}}}\circ f^{l_\beta}\right)
\right)\\
&\leq& \mathbb{E}\left(\left(\prod_{\alpha=1}^r A_{\rho_{2^i}}^+ \circ f^{k_\alpha}\right)
\left(\prod_{\beta=1}^r A_{\rho_{2^j}}^+\circ f^{l_\beta}\right)\right)\\
&=& \mathbb{E}\left(B\left(\prod_{\beta=1}^r A_{\rho_{2^j}}^+\circ f^{l_\beta} \right)
\right)\leq C\left(\left(\EXP B\right) \left(\mathbb{E} A_{\rho_{2^j}}^+\right)^r+C L^{r2^{i+1}}\rho_{2^i}^{-r\tau}
\rho_{2^j}^{-r \tau}\theta^{2^j \eps}\right).\\
\Eea
By applying the  established $(GM1)_r$ to estimate $\mathbb{E}B,$ we note that  the second term is bounded above by $\DS C (L^{r2^{-b+1}})^{2^j} 2^{2r\tau u j} \theta^{2^j \eps}.$ For sufficiently large $b,$ this term is much smaller than the first term, leading us to conclude $(GM3)_r.$

Next, we analyze the composite admissible targets $\left\{\bfE_{\rho_n}\right\}.$ Consider
$$ B^*(x, x_1, x_2\dots x_r)=
\left(\prod_{\alpha=1}^r 1_{\bar{E}_{\rho_{2^i}}^{k_\alpha}}(x)\right)
\left(\prod_{\beta=1}^r 1_{\bar{E}_{\rho_{2^j}}}(x, x_\beta)\right)$$
and
$$ \tB(x,x_1,\cdots,x_r)=\left(\prod_{\alpha=1}^r\bA_{\rho_{2^i}}^+(x,f^{k_\alpha}x)\right)\bA_{\rho_{2^j}}^+(x,x_1)\cdots \bA_{\rho_{2^j}}^+(x,x_r).$$
By ${\rm(\overline{Appr})}$ and ${\rm{(GEM)}_r}$ and the already established $(M1)_r$, we have
\Bea
\mu\Big(\bigcap_{1\leq \alpha,\,\beta \leq r}\big(\bar{E}_{\rho_{2^i}}^{k_\alpha}\cap\bar{E}_{\rho_{2^j}}^{l_\beta}\big)\Big)
&=&\int\left(B^*(x, f^{l_1} x, \dots, f^{l_r} x)\right)\rd\mu(x)\\
&\leq&\int\left(\tB(x, f^{l_1} x, \dots, f^{l_r} x)\right)\rd\mu(x)\\
&\leq&C\left(\bar{\sigma}(\rho_{2^j})^r\int\left(\prod_{\alpha=1}^r \bar{A}^+_{\rho_{2^i}}(x, f^{k_\alpha } x)\right)\rd\mu(x)+L^{r2^{i+1}}\rho_{2^i}^{-r\tau}\rho_{2^j}^{-r \tau}\theta^{2^j \eps}\right). \\
\Eea
Using $(GM1)_r$ again we observe that
$$ \int\Big(\prod_{\alpha=1}^r \bar{A}^+_{\rho_{2^i}}(x, f^{k_\alpha } x)\Big)\rd\mu(x)\leq
C\bar{\sigma}(\rho_{2^i})^r,$$
which allows to complete the proof of $(GM3)_r$ in the case of composite admissible targets.\end{proof}

\section*{ACKNOWLEDGEMENTS}
The author would like to thank Dmitry Dolgopyat for helpful discussions on multiple recurrence  of dispersing billiards.
 This work was supported by  Natural Science Foundation of Beijing No.~1244041.

\end{document}